\documentclass[11pt]{article}
\usepackage[left]{lineno}
\usepackage{blindtext}
\usepackage{amstext, amsmath,latexsym,amsbsy,amssymb}
\usepackage{enumitem}

\newtheorem{proposition}{Proposition}[section]
\newtheorem{remark}{Remark}[section]
\newenvironment{proof}{{\bf Proof.\ }}{\QED\\}
\newtheorem{lemma}{Lemma}[section]

\numberwithin{equation}{section}
\newtheorem{theorem}{Theorem}[section]
\newtheorem{corollary}{Corollary}[section]
\newcommand{\QED}{\hspace*{\fill}\rule{2.5mm}{2.5mm}}

\usepackage{amssymb}\usepackage{graphicx}
\usepackage{tikz}
\usepackage{color}
\newcommand\qed{\hfill$\sqcap\kern-7.5pt\hbox{$\sqcup$}$}

\newcommand{\RR}{\mathbb{R}}

\newcommand{\beqn}{\begin{equation}}
\newcommand{\eeqn}{\end{equation}}
\newcommand{\bear}{\begin{eqnarray}}
\newcommand{\eear}{\end{eqnarray}}
\newcommand{\bean}{\begin{eqnarray*}}
\newcommand{\eean}{\end{eqnarray*}}

\allowdisplaybreaks

\begin{document}

\linenumbers

\title{The Cauchy problem and BEC stability for the quantum Boltzmann-Condensation system for bosons at very low temperature}

\author{Ricardo Alonso*, Irene M. Gamba**, Minh-Binh Tran ***\\
*Department of Mathematics \\
Pontificia Universidad Catolica at Rio de Janeiro\\
email: ralonso@mat.puc-rio.br\\
**Department of Mathematics \\
University of Texas Austin\\
email: gamba@math.utexas.edu\\
**Department of Mathematics \\
University of Wisconsin-Madison\\
email: mtran23@wisc.edu}

\maketitle

\begin{abstract} 
We study a quantum Boltzmann-Condensation system that describes the evolution of the interaction between a well formed Bose-Einstein condensate and the quasi-particles cloud.  The kinetic model is valid for a dilute regime at which the temperature of the gas is very low compared to the Bose-Einstein condensation critical temperature.  In particular, our system couples the density of the condensate from a Gross-Pitaevskii type equation to the kinetic equation through the dispersion relation in the kinetic model and the corresponding transition probability rate from pre to post collision momentum states. We rigorously show the following three properties (1) the well-posedness of the Cauchy problem for the system in the case of a radially symmetric initial configuration, (2) find qualitative properties of the solution such as instantaneous creation of exponential tails and, (3) prove the uniform condensate stability related to the initial mass ratio between condensed particles and quasi-particles.  The stability result from (3) leads to global in time existence of the initial value problem for the quantum Boltzmann-Condensation system.  
\end{abstract}



\tableofcontents
\section{Introduction}
After the first Bose-Einstein Condensate (BEC) was produced by Cornell, Wieman, and Ketterle \cite{WiemanCornell,Ketterle}, there has been an immense amount of research on BECs and cold bosonic gases.  Above the condensation temperature, the dynamic of a bose gas is determined by the Uehling-Uhlenbeck kinetic equation introduced in \cite{UehlingUhlenbeck:TPI:1933}; see for instance \cite{EscobedoVelazquez:2015:FTB} for interesting results and list of references.  The first proof of BECs was done in \cite{lieb2002proof}. Below the condensation temperature, the bosonic gas dynamics is  governed by a system that couples a quantum Boltzmann  and a Gross-Pitaevskii equations. In such a system, the wave function of the BEC follows the Gross-Pitaevskii equation and the quantum Boltzmann equation describes the evolution of the density function of the excitations (quasi-particles). The system   was first derived by Kirkpatrick and Dorfmann in \cite{KD1,KD2}, using a Green function approach and  was revisited by Zaremba-Nikuni-Griffin and Gardiner-Zoller et. al. in \cite{QK1,QK0,ZarembaNikuniGriffin:1999:DOT}. It  has, then, been developed and studied extensively in the last two decades by several authors from the application perspective (see \cite{bijlsma2000condensate,ColdAtoms1,Stoof:1999:CVI}, and references therein). In \cite{Spohn:2010:KOT}, Spohn gave a heuristic derivation for the one-dimensional version of the system, using a perturbation argument for the Uehling-Uhlenbeck equation. A more formal derivation, for the full three dimensional case, is done in \cite{ReichlTran} where some ideas from the works \cite{ChenHainzl:UUF:2015,ErdosSchleinYau:DOT:2010} were taken together with techniques from quantum field theory.

In this work, we focus on the rigorous mathematical study of the dynamics of dilute Bose gases modeled by the quantum Boltzmann equation at very low temperature  coupled to  the condensation model at the quantum level. The quantum Boltzmann model that we referred to was introduced in \cite{E,EPV,KD1,KD2}, that is, the BEC is well formed and the interaction between excited atoms is secondary relative to the interaction between excited atoms with the BEC.  The condensation at the quantum level may be described by classical models such as Gross-Pitaevskii  \cite{ChenHainzl:UUF:2015,ErdosSchleinYau:DOT:2010,lieb2002proof}.  At this quantum level the BEC mass is given by $n_c=n_c(t):=|\Psi|^2(t)$, where $\Psi$ is the wave function of the quantum condensation satisfying a Gross-Pitaevskii type equation with an absorption term proportional to the averaged of the interacting particle (collision) operator from the quantum kinetic model, and the corresponding quantum probability density of the excited states evolves according to the quantum Boltzmann equation with interacting particle (collision) operator proportional to the condensate $n_c(t)$ (cf. \cite{ArkerydNouri:2012:BCI,Spohn:2010:KOT,PomeauBrachetMetensRica}).

Under these assumptions, the evolution of the space homogeneous probability density distribution function $f:=f(t,p)$, with $(t,p)\in[0,\infty)\times\mathbb{R}^{3}$, for $p$ the momenta state variable, of the excited bosons and the condensate mass $n_{c}:=n_{c}(t)$ can be described by the following Boltzmann-Gross-Pitaevskii system
\begin{equation}\label{QuantumBoltzmann}
\left\{\begin{array}{clc}
\frac{\text{d}f}{\text{d}t} &= n_c\,Q[n_{c}, f]\,, & f(0,\cdot)=f_0\,,\vspace{0.2cm}\\ 
\frac{\text{d}n_{c}}{\text{d}t} &= -n_c \int_{\mathbb{R}^{3}}\text{d}p\,Q[n_{c}, f]\,, & n_{c}(0)=n_{0}\,,
\end{array}\right.
\end{equation}
where the interaction operator is defined as
\begin{align}\label{E1}
\begin{split}
Q[n_{c},f]:=\int _{ \RR^3}&\int _{ \RR^3} \text{d}p_1\text{d}p_2 \big[R(p, p_1, p_2)-R(p_1, p, p_2)-R(p_2, p_1, p) \big]\,,\\
R(p, p_1, p_2):=\\
&\hspace{-1.5cm}|\mathcal M(p, p_1, p_2)|^2\big[\delta \left(  \frac{\omega (p)}{k_BT} -\frac{\omega (p_1)}{k_BT}-\frac{\omega (p_2)}{k_BT}  \right)  \delta (p-p_1-p_2)\big]\\
&\hspace{-.5cm}\times\big[ f(p_1)f(p_2)(1+f(p))-(1+f(p_1)(1+f(p_2))f(p) \big]\,,
\end{split}
\end{align}
where $\beta :=\frac {1} {k_B T}>0$ is a physical constant depending on the Boltzmann constant $k_B$, and the temperature of the quasiparticles $T$ at equilibrium.  The particle energy $\omega (p)$ is given by the Bogoliubov dispersion law
\bear
\omega (p)=\left[\frac {g n_c} {m}|p|^{2}+\left(\frac {|p|^2} {2m} \right)^2 \right]^{1/2}, \label{E3}
\eear
where $p\in\mathbb{R}^{3}$ is the momenta, $m$ is the mass of the particles, $g$ is an interaction ``excited-condensate'' coupling constant and $n_c$ is the condensate mass, as introduced earlier.

The term $\mathcal M(p, p_1, p_2)$ is referred as the transition probability or matrix element (as much as collision kernel).  Its constitutive relation depends on the dispersion relation $\omega (p)$ and, consequently, strongly couples the quantum Boltzmann equation to the quantum condensate.

In the regime treated in this document, the transition probability can be approximated up to first order to a workable expression.  Indeed, we restrict the range of the temperature $T$, the condensate density $n_c$, and the interaction coupling constant $g$ to values  for which  $k_BT$ is much smaller than $(gn_c/m)^{1/2}$, i.e. a {\em cold gas regime}.  Under this condition, the dispersion law $\omega (p)$  in \eqref{E3} is approximated by 
\begin{equation*}\label{E6}
\frac1{k_BT} \left[\frac {g n_c} {m}|p|^{2}+\left(\frac {|p|^2} {2m} \right)^2 \right]^{1/2} \approx  \frac c{k_BT}|p|, \quad \text{where }\;c:=\sqrt{\frac {gn_c} {m}} \, , 
\end{equation*}
as long as  $|p| \ll 2\sqrt{gn_cm}$. In particular, the energy will now be defined by the phonon dispersion law (still using the same notation), see \cite{E,IG} 
\begin{equation}\label{E6}
\omega (p)=c |p|, \quad \text{for}\;c:=c(t) =\sqrt{\frac {gn_c(t)} {m}}\, .
\end{equation}
Under the {\em cold gas regime}, the  transition probability $\mathcal M$ is   approximated by (see, for instance \cite[eq. (7)]{EPV}, \cite[eq. (83)]{IG}, \cite[eq. (42)]{E})
\begin{equation}\label{E7}
|\mathcal M|^2=\kappa|p||p_1||p_2|\,
\end{equation}
where 
\begin{equation}\label{E7a}
\kappa=\frac {9c}{64\pi ^2mn_c^2} = \frac{9 }{64\pi^{2}(mg n_{c})^{3/2}}.
\end{equation}

Note that the transition probability could also be approximated as (cf. \cite{ArkerydNouri:2012:BCI})
$$|\mathcal M|^2=\frac{\omega(p)\omega(p_1)\omega(p_2)}{32g^3n_c^3}.$$

We perform the analysis in the whole momentum space, not in a piece of it or the torus \cite{Spohn:TPB:2006}, requiring a detailed control of the solution's tails and low temperature behavior.  

Using that $\delta(\cdot)$ is homogeneous of degree $-1$, the reduced phonon dispersion law \eqref{E6} is implemented as $\delta(c|p|)=c^{-1}\delta(|p|)$, and so the quantum collisional integral \eqref{E1} becomes
\begin{align}\label{EE1}
\begin{split}
Q[n_{c},f]:&=\kappa c^{-1}\int _{ \RR^3}\int _{ \RR^3}\text{d}p_1\text{d}p_2 \big[R(p, p_1, p_2)-R(p_1, p, p_2)-R(p_2, p_1, p) \big]\\
& R(p, p_1, p_2):=  \mathcal{K}(|p|,|p_1|,|p_2|)\big[\delta \left(   |p|-|p_1|-|p_2|    \right)  \delta (p-p_1-p_2)\big]\\
&\quad\times\big[ f(p_1)f(p_2)(1+f(p))-(1+f(p_1)(1+f(p_2))f(p) \big]\,.
\end{split}
\end{align}
Here we introduced $\mathcal{K}(|p|,|p_{1}|,|p_{2}|):=|p||p_{1}||p_{2}|$.  Clearly, from the interaction law $p=p_1 +p_2$ and $|p| =  |p_1| +|p_2|$ modeled in the collision operator by the singular Dirac delta masses, this trilinear collisional form \eqref{EE1} is reduced into a bilinear one, that can be split in the difference of two positive quadratic operators, as will be shown in the existence result.
\\
In addition,  the low temperature quantum collisional form \eqref{EE1} can be split into  {\em gain} and  {\em loss} operator forms
\begin{align}\label{GainLossSplit}
\begin{split}
Q[n_{c},&f](t,p) = Q^+[n_{c},f](t,p)  - Q^-[n_{c},f](t,p) \\
& = \kappa c^{-1}\Big(\mathcal{Q}^+[f](t,p) - f(t,p)\, \nu[f](t,p)\Big)=:\kappa c^{-1}\mathcal{Q}[f](t,p), 
\end{split}
\end{align}
as is done with the classical Boltzmann operator.  Here, the gain operator is also defined by the positive contributions in the total  rate of change in time of the collisional form $Q[n_{c},f](t,p)$ in \eqref{EE1}, that is, $Q^+[n_{c},f]=\kappa c^{-1}\mathcal{Q}^+[f]$ where 
\begin{align}\label{GainOperator}
\begin{split}
&\mathcal{Q}^+[f](t,p):=\int_{\mathbb{R}^3}\int_{\mathbb{R}^3}\text{d}p_1\text{d}p_2\,\mathcal{K}(|p|,|p_1|,|p_2|)\delta(p-p_1-p_2)\\
&\hspace{.3cm}\times\delta(|p|-|p_1|-|p_2|)f(t,p_1)f(t,p_2)+2\int_{\mathbb{R}^3}\int_{\mathbb{R}^3}\text{d}p_1\text{d}p_2\,\mathcal{K}(|p|,|p_1|,|p_2|)\\
&\hspace{.8cm}\times\delta(p_1-p-p_2)\delta(|p_1|-|p|-|p_2|)\big[2f(t,p)f(t,p_1)+f(t,p_1)\big]\,.
\end{split}
\end{align}
Similarly, the loss operator models the negative contributions in the total  rate of change in time of same collisional form $Q[n_{c},f](t,p)$. It is local in $f(t,p)$ and so  written $Q^-[n_{c},f] := \kappa c^{-1}\,f\, \nu[f]$, where $\nu[f](t,p)$, referred as the {\em collision frequency or attenuation coefficient}, is defined by 
\begin{align}\label{LossOperator}
\begin{split}
\nu&[f](t,p):=\int_{\mathbb{R}^3}\int_{\mathbb{R}^3}\text{d}p_1\text{d}p_2\,\mathcal{K}(|p|,|p_1|,|p_2|)\delta(p-p_1-p_2)\\
&\times\delta(|p|-|p_1|-|p_2|)\big[2f(t,p_1)+1\big]+2\int_{\mathbb{R}^3}\int_{\mathbb{R}^3}\text{d}p_1\text{d}p_2\,\mathcal{K}(|p|,|p_1|,|p_2|)\\
&\hspace{1cm}\times\delta(p_1-p-p_2)\delta(|p_1|-|p|-|p_2|)f(t,p_2)\,,
\end{split}
\end{align}
and it is nonlocal in $f(t,p)$.  Note that the collisional operator $\mathcal{Q}[f]:=\mathcal{Q}^{+}[f] - f\,\nu[f]$ is independent of $n_{c}$. 

In summary, our goal is to study the Cauchy problem of radial solutions for the  Boltzmann-Gross-Pitaevskii system \eqref{QuantumBoltzmann} at low temperature, which, with the definitions of \eqref{GainLossSplit}, \eqref{GainOperator} and \eqref{LossOperator}, reads
\begin{equation}\label{BGP}
\left\{\begin{array}{clc}
\frac{\text{d}f}{\text{d}t} &= \tfrac{\kappa_{0}}{n_{c}}\mathcal{Q}[f]\,, &f(0,\cdot)=f_0\,,\vspace{0.2cm}\\
\frac{\text{d}n_{c}}{\text{d}t} &= -\tfrac{\kappa_{0}}{n_{c}}\int_{\mathbb{R}^{3}}\text{d}p\,\mathcal{Q}[f]\,, & n_{c}(0)=n_{0}\,,
\end{array}\right.
\end{equation}
where the resulting constant $\kappa_{0}=\frac{9}{64\pi^{2}m}$.\\

The organization of the paper is as follows.
\begin{itemize}
\item [$\cdot$] In section 2 and 3 we present the weak and strong formulations of the collision operator and use them to recall the main conservation laws as well as the entropy estimate corresponding to an $H$-Theorem for \eqref{QuantumBoltzmann} in the low temperature regime collisional form \eqref{EE1}.
\item[$\cdot$] Section 4 considers \textit{a priori} estimates on the observables or moments of solutions.  These are related to high energy tail behavior and will be developed in context of radially symmetric solutions.  Moment propagation techniques have been developed for the classical Boltzmann equation in \cite{AlonsoGamba:2013:ANA,GambaPanferov:2004:OTB,AlonsoGamba:2015:OML}.
\item [$\cdot$] In section 5 we address the central issue of the BEC stability.  It is clear that the condition $n_{c}>0$ is essential for the validity of the approximations that have been made in the derivation of the model.  In this section we take advantage of the nonlinear nature of the equation to derive $L^{\infty}$-estimates that allow us to show the BEC uniform stability.  Natural conditions in terms of the ratio between the initial mass of the condensate and quasi-particles are necessary for the sustainability of the condensate in the long run.  This result formalizes the validity of the decomposition of the total density of the gas between a singular part (condensate) and a regular part (quasi-particles) and leads to global in time well-posedness of the problem.
\item [$\cdot$] The existence and uniqueness arguments given in section 6 are based on the \textit{a priori} estimates on the solution's moments and the $L^{\infty}$-estimate provided for BEC stability.  When such estimates are combined with classical abstract ODE theory, the result is a robust and elegant technique to prove well-posedness for collisional integral equations.
\item [$\cdot$] Finally, in section 7, we show that solution to the Cauchy problem have exponential decaying tails in the sense of $L^1(\mathbb{R}^3)$, which are referred to as Mittag-Leffler tails that were introduced for the Boltzmann equation in \cite{AlonsoGamba:2015:OML}.  This result formalizes, at least qualitatively, the approximations that are made in the low temperature regime were narrow distribution profiles are assumed.
\end{itemize}

\section{Weak and strong formulation of collisional forms}

The following properties hold for the low temperature quantum collisional form \eqref{EE1} remarking that, for notational convenience, we will usually omit the time variable $t$ unless some stress is necessary in the context.   

\begin{proposition}[Weak Formulation]\label{Lemma:WeakFormulation}
For any suitable test function $\varphi$, the following weak  formulation holds for the collision operator $\mathcal{Q}$
\begin{align}\label{Lemma:WeakFormulation:Eq1}
\begin{split}
\int_{\mathbb{R}^3}& dp\, \mathcal{Q}[f](p)\varphi(p)=\int_{\mathbb{R}^3}\int_{\mathbb{R}^3}\int_{\mathbb{R}^3}dp\,dp_1\,dp_2\, \mathcal{K}(|p|,|p_1|,|p_2|)\delta(p-p_1-p_2)\\
&\times \delta(|p|-|p_1|-|p_2|)\Big[ f(p_1)f(p_2)-f(p_1)f(p)-f(p_2)f(p)-f(p) \Big]\\
&\hspace{1cm}\times \Big[\varphi(p)-\varphi(p_1)-\varphi(p_2)\Big]\\
&\hspace{-.7cm}=2\pi\int_{\mathbb{R}^3}dp_1\int_{\mathbb{R}^+}|p_2|^{2}d|p_2|\,\frac{|p_{1}|+|p_{2}|}{|p_{1}||p_{2}|}\mathcal{K}\big(|p_1|+|p_2|, |p_1|, |p_2|\big)\Big[ f(p_1)f(|p_2|\widehat{p_1}) \\
& -f(p_1)f(p_1+|p_2|\widehat{p_1})-f(|p_2|\widehat{p_1})f(p_1+|p_2|\widehat{p_1})-f(p_1+|p_2|\widehat{p_1}) \Big]\\
&\times \Big[\varphi(p_1+|p_2|\widehat{p_1})-\varphi(p_1)-\varphi(|p_2|\widehat{p_1})\Big]\,,
\end{split}
\end{align}
As a consequence, for radially symmetric functions $f(p):=f(|p|)$ and $\varphi(p):=\varphi(|p|)$, the following holds true
\begin{align}\label{WeakFormulation:radial}
\begin{split}
\int_{\mathbb{R}^3}&dp\,\mathcal{Q}[f](p)\varphi(p) =8\pi^{2}\int_{\mathbb{R}_+}\int_{\mathbb{R}_+}d|p_1|\,d|p_2|\,\mathcal{K}_0\big(|p_{1}| + |p_{2}|,|p_{1}|,|p_{2}|\big) \times \\
&\hspace{-.2cm}\Big[ f(|p_1|)f(|p_2|) - f(|p_1|)f(|p_1|+|p_2|) - f(|p_2|)f(|p_1|+|p_2|) \\
& - f(|p_1|+|p_2|) \Big] \times \Big[\varphi(|p_1|+|p_2|)-\varphi(|p_1|)-\varphi(|p_2|)\Big]\,,
\end{split}
\end{align}
where $\mathcal{K}_{0}(|p|,|p_{1}|,|p_{2}|):=|p||p_{1}||p_{2}|\mathcal{K}(|p|,|p_{1}|,|p_{2}|)=|p|^{2}|p_{1}|^{2}|p_{2}|^{2}$.
\end{proposition}
\begin{proof}
In this proof we use the short-hand $\int:=\int_{\mathbb{R}^{9}}\text{d}p\,\text{d}p_1\,\text{d}p_2$.  First, observe that 
\begin{align}\label{Lemma:WeakFormulation:Eq2}
\begin{split}
&\int_{\mathbb{R}^3}\text{d}p\,\mathcal{Q}[f](p)\varphi(p) = \\
&\hspace{.3cm}\int  \mathcal{K}(|p|,|p_1|,|p_2|)\delta(p-p_1-p_2)\delta(|p|-|p_1|-|p_2|)R(p, p_1, p_2)\varphi(p)\\
&\hspace{.6cm}-\int \mathcal{K}(|p|,|p_1|,|p_2|)\delta(p-p_1-p_2)\delta(|p|-|p_1|-|p_2|)R(p_1, p, p_2)\varphi(p)\\
&\hspace{.9cm}-\int  \mathcal{K}(|p|,|p_1|,|p_2|)\delta(p-p_1-p_2)\delta(|p|-|p_1|-|p_2|)R(p_2, p_1, p)\varphi(p)\,.
\end{split}
\end{align}
Second, interchanging variables $p\leftrightarrow p_1$ and $p\leftrightarrow p_2$,
\begin{equation}\label{Lemma:WeakFormulation:Eq3}
\int  \mathcal{K}(|p|,|p_1|,|p_2|) 
R(p_1, p, p_2)\varphi(p) = \int  \mathcal{K}(|p|,|p_1|,|p_2|) R(p, p_1, p_2)\varphi(p_1)\,,
\end{equation}
and
\begin{equation}\label{Lemma:WeakFormulation:Eq4}
\int  \mathcal{K}(|p|,|p_1|,|p_2|) R(p_2, p_1, p)\varphi(p) = \int  \mathcal{K}(|p|,|p_1|,|p_2|) R(p, p_1, p_2)\varphi(p_2)\,.
\end{equation}
Combining \eqref{Lemma:WeakFormulation:Eq2}, \eqref{Lemma:WeakFormulation:Eq3}, \eqref{Lemma:WeakFormulation:Eq4}, we get the first equality in \eqref{Lemma:WeakFormulation:Eq1}.  Now, evaluate the Dirac in $p=p_{1}+p_2$ (conservation of momentum) to obtain
\begin{align}\label{Lemma:WeakFormulation:Eq5}
\begin{split}
\int_{\mathbb{R}^3}\text{d}p&\,\mathcal{Q}[f](p)\varphi(p)= \int_{\mathbb{R}^3}\int_{\mathbb{R}^3}\mathcal{K}(|p_{1}+p_{2}|,|p_1|,|p_2|)\delta(|p_1+p_2|-|p_1|-|p_2|)\\
&\Big[ f(p_1)f(p_2)-f(p_1)f(p_1+p_2)-f(p_2)f(p_1+p_2)-f(p_1+p_2) \Big]\\
& \hspace{1cm}\times \Big[\varphi(p_1+p_2)-\varphi(p_1)-\varphi(p_2)\Big]\text{d}p_{1}\,\text{d}p_2\,,
\end{split}
\end{align}
Now, observe that $|p_1+p_2|-|p_1|-|p_2|=0$ if and only if $\widehat{p_1}\cdot\widehat{p_2}=1$.  Since,  $$|p_1+p_2|-|p_1|-|p_2|=\big(|p_{1}|^{2} + |p_{2}|^{2} + 2|p_{1}||p_{2}|\widehat{p_1}\cdot\widehat{p_2}\big)^{1/2}-|p_{1}|-|p_{2}|\,,$$ it follows from a polar change of variable, taking $\widehat{p_1}$ as the zenith, that the following identity holds for any continuous function $F(p_2)$
\begin{align*}
\int_{\mathbb{R}^{3}}&\text{d}p_2\,F(p_{2})\,\delta\big(|p_1+p_2|-|p_1|-|p_2|\big)\\
&=\int_{\mathbb{R}^{+}}|p_2|^{2}\text{d}|p_2|\int^{2\pi}_{0}\text{d}\phi\int^{1}_{-1}\text{d}s\, F\big(p_{2}(s,\sin(\phi))\big)\delta(y(s))\\
&=2\pi\int_{\mathbb{R}^{+}}|p_2|^{2}\text{d}|p_2|\,\frac{F(|p_{2}|\widehat{p_1})}{y'(1)} = 2\pi\int_{\mathbb{R}^{+}}|p_2|^{2}\text{d}|p_2|\,F(|p_{2}|\widehat{p_1}) \frac{|p_1|+|p_{2}|}{|p_{1}||p_{2}|}\,,
\end{align*}   
where $y(s)=\big(|p_{1}|^{2} + |p_{2}|^{2} + 2|p_{1}||p_{2}|s\big)^{1/2}-|p_1|-|p_2|$.
In the second identity we used that $p_{2}(1,\sin(\phi))=|p_2|\widehat{p_1}$ and, for the latter, the fact that $y'(1)=\frac{|p_{1}|p_{2}|}{|p_{1}| + |p_{2}|}$.  Using this identity in \eqref{Lemma:WeakFormulation:Eq5} proves the second equality in \eqref{Lemma:WeakFormulation:Eq1}.  Finally, for radially symmetric functions $f(p):=f(|p|)$ and $\varphi(p):=\varphi(|p|)$, one simply uses that $\big|p_1 + |p_2|\widehat{p_1}\big| = |p_1| + |p_2|$ and polar coordinates in the $p_{1}$-integral to obtain \eqref{WeakFormulation:radial}  
\end{proof}
Based on the weak formulation of the collision operator, we can deduce its strong formulation.  The strong formulation will be important for finding $L^{\infty}$-estimates to prove the BEC uniform stability.  The nonlinear part of the operator will play an important role in the estimates, thus, in this context we write the operator as a quadratic part and a linear part
\begin{equation*}
\mathcal{Q}[f](p) = \mathcal{Q}_{q}[f](p) + \mathcal{L}[f](p)\,,
\end{equation*}
and stress that this decomposition is different from that of gain and loss parts.  Indeed, the linear part is only a piece of the loss operator which includes bilinear terms.
\begin{corollary}[Strong Formulation]\label{Lemma:StrongFormulation}
Let $f$ be a radially symmetric function.  The strong formulation of the collision operator consists in 9 quadratic terms, namely,
\begin{align*}
\mathcal{Q}_{q}[f]&(|p|) : = 8\pi^{2}\bigg(\int^{|p|}_{0}d|p_{1}|K(|p_{1}|,|p|-|p_{1}|)f(|p_{1}|)f(|p|-|p_{1}|)\\
& + \int^{\infty}_{|p|}d|p_{1}|\big(K(|p_{1}|-|p|,|p|)+K(|p|,|p_{1}|-|p|)\big)f(|p_{1}|)f(|p_{1}|-|p|)\bigg)\\
&\hspace{-1cm} + 8\pi^{2}\,f(|p|)\bigg(\int^{\infty}_{|p|}d|p_{1}|\big(K(|p|,|p_{1}|-|p|) + K(|p_{1}|-|p|,|p|)\big)f(|p_{1}|)\\
& - \int^{|p|}_{0}d|p_{1}|\big(K(|p|-|p_{1}|,|p_{1}|) + K(|p_{1}|,|p|-|p_{1}|)\big)f(|p_{1}|)\\
&- \int^{\infty}_{0}d|p_{1}|\big(\mathcal{K}_{0}(|p|,|p_{1}|) + \mathcal{K}(|p_{1}|,|p|)\big)f(|p_{1}|)\bigg)\,.
\end{align*}
The strong formulation of the linear operator reduces to 3 terms,
\begin{align*}
\mathcal{L}[f](|p|) = 8\pi^{2}&\bigg(\int^{\infty}_{|p|}d|p_{1}|\big(K(|p|,|p_{1}|-|p|)+K(|p_{1}|-|p|,|p|)\big)f(|p_{1}|)\\
& - f(|p|)\int^{|p|}_{0}d|p_{1}|K(|p_{1}|,|p|-|p_{1}|)\bigg)\,,
\end{align*}
where the symmetric collision kernel is defined by $$K(|p_{1}|,|p_{2}|):=\mathcal{K}_{0}(|p_{1}|+|p_{2}|,|p_{1}|,|p_{2}|)=|p_{1}|^{2}|p_{2}|^{2}(|p_{1}| + |p_{2}|)^{2}\,.$$  In these expressions we included the polar Jacobian for notational simplicity.
\end{corollary}
\begin{proof}
The strong formulation follows by a simple, yet tedious, calculation involving change of variables.  For instance, take the first term in the radial weak formulation \eqref{WeakFormulation:radial}
\begin{align*}
\int_{\mathbb{R}_+}&\int_{\mathbb{R}_+}\text{d}|p_1|\,\text{d}|p_2|\,K\big(|p_{1}|,|p_{2}|\big)\,f(|p_1|)f(|p_2|)\varphi(|p_1|+|p_2|)=\\
&\int_{\mathbb{R}_+}\text{d}|p|\varphi(|p|)\Big(\int^{|p|}_{0}\text{d}|p_1|\,K\big(|p_{1}|,|p|-|p_{1}|\big)\,f(|p_1|)f(|p|-|p_{1}|)\Big)\,.
\end{align*}
Since this identity is valid for any suitable test function $\varphi$, one obtains the term
\begin{equation*}
\int^{|p|}_{0}\text{d}|p_1|\,K\big(|p_{1}|,|p|-|p_{1}|\big)\,f(|p_1|)f(|p|-|p_{1}|)
\end{equation*}
in the strong formulation.  Other terms are left to the reader.
\end{proof}
\section{Conservation of laws and $H$-Theorem}
The weak formulation presented in Proposition \ref{Lemma:WeakFormulation} implies the following conservation laws and a quantum version of the classical Boltzmann $H$-Theorem.
\begin{corollary}[Conservation laws]
If $(f,n_{c})$ is a solution of the system \eqref{QuantumBoltzmann}, it formally conserves mass, momentum and energy 
\begin{eqnarray}
\label{Coro:ConservatioMass}
\int_{\mathbb{R}^3}dp\,f(t,p) \ + \  n_c(t)&=&\int_{\mathbb{R}^3}dp\, f_0(p) \ + \  n_c(0)\\
\label{Coro:ConservatioMomentum}
\int_{\mathbb{R}^3}dp\,f(t,p)\,p&=&\int_{\mathbb{R}^3}dp\,f_0(p)\,p\,,\\
\label{Coro:ConservatioEnergy}
\int_{\mathbb{R}^3}dp\,f(t,p)\,|p|&=&\int_{\mathbb{R}^3}dp\,f_0(p)\,|p|\,.
\end{eqnarray}
\end{corollary}
\begin{remark}
Since $f$ is the density related to the thermal cloud only, the mass is not conserved for $f$ but for the total density $f+n_{c}\,\delta(p)$.  Of course, particles enter and leave the condensate at all times.
\end{remark}
\begin{corollary}[H-Theorem] If  $f(t,p)$ solves \eqref{QuantumBoltzmann}, then 
\begin{equation*}
\frac{d}{dt}\int_{\mathbb{R}^3}dp\Big[f(p)\log f(p)-\big(1+f(p)\big)\log\big(1+f(p)\big)\Big]\leq 0.
\end{equation*}
As a consequence, a radially symmetric equilibrium of the equation has the form
\begin{equation}\label{def-equilibrium}
f_{\infty}(p)=\frac{1}{e^{\alpha\,\omega(p)}-1}, \quad \text{for some }\; \alpha>0.
\end{equation}
This distribution is usually referred as a Bose-Einstein distribution.
\end{corollary}
\begin{remark}
The linearization of the equation \eqref{QuantumBoltzmann} about Bose-Einstein states can be performed by setting
\begin{equation*}
f(t, p)=f_\infty(p)+f_\infty(p)\big(1+f_\infty(p)\big)\Omega (t, p).
\end{equation*}
After plugging into the collision operator and neglecting the nonlinear terms, one has
\begin{equation*}
f_\infty(p)\big(1+f_\infty(p)\big)\frac {\partial \Omega } {\partial t}(t, p)=-M(p)\Omega  (t, p)+\int  _{ \RR^3 }
 dp'\,\mathcal U(p, p') \Omega (t, p'),
\end{equation*}
for some explicit function $M(p)$ and measure $\mathcal U(p, p')$. We refer to \cite{CraciunBinh,EscobedoBinh},  for the study of this equation in this perturbative setting and further discussions on this direction.
\end{remark}
\begin{proof}
We observe that
\begin{align*}
\frac{\text{d}}{\text{d}t}\int_{\mathbb{R}^3}\text{d}p&\Big[f(p)\log f(p) - \big(1+f(p)\big)\log\big(1+f(p)\big)\Big]=\\
&\int_{\mathbb{R}^3}\text{d}p\,\partial_t f(p)\log\bigg(\frac{f(p)}{f(p)+1}\bigg).
\end{align*}
In addition, we can rewrite
\begin{align*}
&\int_{\mathbb{R}^3}\text{d}p\,\mathcal{Q}[f](p)\varphi(p)=\int_{\mathbb{R}^9}\mathcal{K}(|p|,|p_1|,|p_2|)\delta(p-p_1-p_2)\delta(|p|-|p_1|-|p_2|)\\
&\hspace{2cm}\times \big(1+f(p)\big)\big(1+f(p_1)\big)\big(1+f(p_2)\big)\\
&\times \bigg( \frac{f(p_1)}{f(p_1)+1}\frac{f(p_2)}{f(p_2)+1}-\frac{f(p)}{f(p)+1} \bigg)\Big[\varphi(p)-\varphi(p_1)-\varphi(p_2)\Big]\text{d}p\text{d}p_1\text{d}p_2.
\end{align*}
Choosing  $\varphi(p)=\log\left(\frac{f(p)}{f(p)+1}\right)$ we obtain, in the case of equality, that
$$\frac{f(p_1)}{f(p_1)+1}\frac{f(p_2)}{f(p_2)+1}-\frac{f(p)}{f(p)+1}=0,$$
or equivalently, putting   
$h(p)=\log \left(\frac{f(p)}{f(p)+1}\right)$, we get
\begin{equation}\label{id-h}h(p_1)+h(p_2)=h(p).\end{equation}
The fact that $h(\cdot)$ is radially symmetric yields  $h(p) =  - \alpha\,{\omega}(p)$, for all $p \in \mathbb{R}^3$ and some positive constant $\alpha$.  This proves the claim.
\end{proof}
\section{\textit{A priori} estimates on a solution's moments}
The aim of the following sections is to consider radially symmetric solutions of \eqref{QuantumBoltzmann}-\eqref{EE1} that lie in $\mathcal{C}\big([0,\infty);L^{1}(\mathbb{R}^{3},|p|^{k}\text{d}p)\big)$ where
\begin{equation*} 
L^1(\mathbb{R}^{3},|p|^{k}\text{d}p):=\Big\{f\;\text{measurable}\; \big| \;  \int_{\mathbb{R}^{3}} \text{d}p\,|f(p)||p|^k <\infty,\; k\geq0\Big\}.
\end{equation*}
That is, in the sections $4$ and $5$ the \textit{a priori} estimates \textit{assume} the existence of a radially symmetric solution $f(t,\cdot)$ enjoying time continuity in such Lebesgue spaces (thus, time continuity for such solution's moments), for $k$ sufficiently large, say $0\leq k \leq 5$.  Define the solution's moment of order $k$ as 
\begin{equation}\label{Def:MomentsFull}
\mathcal{M}_{k}\langle f \rangle(t) := \int_{\mathbb{R}^3} \text{d}p f(t,p)|p|^{k}\,.
\end{equation}
When $f$ is as radially symmetric function $f(t,p)=f(t,|p|)$, one can use spherical coordinates to reduce the integral with respect to $\text{d}p$ on $\mathbb{R}^3$ to an integral on $\mathbb{R}_+$ with respect to $\text{d}|p|$.  As a consequence,
\begin{equation*}
\mathcal{M}_{k}\langle f \rangle(t) = \big|\mathbb{S}^{2}\big|\int_{\mathbb{R}^+} \text{d}|p| f(t,|p|)|p|^{k+2}\,.
\end{equation*}
Thus, it will be convenient for notation purposes to introduce and work with what we call ``line-moments'' 
\begin{equation}\label{Def:MomentsLine}
m_{k}\langle f \rangle(t) : = \int_0^\infty \text{d}|p|\,f(t,|p|)|p|^{k}\,.
\end{equation}
Observe that $\mathcal{M}_{k}\langle f \rangle = |\mathbb{S}^{2}|m_{k+2}\langle f \rangle$. 

We are going to use the definition of moments in two contexts:  In one hand, in sections 4, 5 and 7 we always consider the moment applied to a given \textit{radial solution of the equation}.  Thus, there is no harm to omit the function dependence and just write $\mathcal{M}_{k}(t)$, $\mathcal{M}_{k}$, $m_{k}(t)$ or $m_{k}$ to denote moments and line-moments for simplicity.  In the other hand, in section 6 we will use moments as norms of the spaces $L^1(\mathbb{R}^{3},|p|^{k}\text{d}p)$, as a consequence, the functional dependence will be important, so we write $m_{k}\langle f \rangle$.  Note that according to the conservation law \eqref{Coro:ConservatioMomentum} and assuming initial energy finite, the following equivalent estimates hold
\begin{equation*}
\mathcal{M}_1(t)=\mathcal{M}_1(0)<\infty\,,\qquad m_3(t)=m_3(0)<\infty\,.
\end{equation*}
Before entering into details, let us explain the necessity of considering radially symmetric solutions of the equation \eqref{QuantumBoltzmann} in the following arguments.  Choosing $\varphi(p)=|p|^{k}$ in the weak formulation Proposition \ref{Lemma:WeakFormulation}, one is lead to estimate terms of the form
\begin{equation*}
\int_{\mathbb{R}^3}\text{d}p_1f(t,p_1)|p_{1}|^i\int_{\mathbb{R}_+}\text{d}|p_2| f(t,|p_2|\widehat{p_1})|p_{2}|^{j}\,,\qquad i,\,j\in\mathbb{N}.
\end{equation*} 
These terms are not estimated by products of moments of $f$ unless the function is radially symmetric.  In such a case this particular term simply writes as a product of line-moments of $f$, namely $|\mathbb{S}^{2}|m_{i+2}\langle f \rangle\,m_{j}\langle f \rangle$.  This technical issue will be central in finding closed \textit{a priori} estimates in terms of line-moments of solutions.
\begin{lemma}\label{control1}
For any suitable function $f\geq0$, for $k\ge 0$, define the quantity
$$\mathcal{J}_k=\int_{\mathbb{R}^{3}} dp\,\mathcal{Q}_{q}[f]|p|^{k},$$
we have:
\begin{itemize}
\item If $k=0$, then
\begin{equation}\label{control1:1}
\mathcal{J}_0\leq C_{k}m_{2}\langle f \rangle\,m_{4}\langle f \rangle.
\end{equation}
\item If $k\geq1$, then 
\begin{equation}\label{control1:2}
\mathcal{J}_k\leq C_{k}\big(m_{k+3}\langle f \rangle\,m_{3}\langle f \rangle + m_{k+1}\langle f \rangle\,m_{5}\langle f \rangle\big).
\end{equation}\end{itemize}

We only prove \eqref{control1:2}, the other inequality \eqref{control1:1} can be proved by the same argument.
The constant $C_{k}>0$ only depends on $k$.  In addition, the linear part simply reads for all $k\ge 0$
\begin{align*}
\int_{\mathbb{R}^{3}} &dp\,\mathcal{L}[f]|p|^{k}= c_{k}\,m_{k+7}\langle f \rangle\,,\text{ with positive constant given by}\\
&\qquad c_{k}=8\pi^{2}\int^{1}_{0}dz\,z^{2}(1-z)^{2}(1-z^{k} - (1-z)^{k})\,.
\end{align*}
\end{lemma}
\begin{proof}
Using the weak formulation \eqref{WeakFormulation:radial}, the pointwise inequality
\begin{equation*}
 0 \leq (x+y)^{k} - x^{k} - y^{k} \leq C_{k}\big(y\,x^{k-1} + y^{k-1}x\big),\,\text{ valid for any }\, k\geq1\,,
 \end{equation*}
and neglecting all the negative contributions, one concludes that
\begin{align*}
&\int_{\mathbb{R}^{3}}\text{d}p\,\mathcal{Q}_{q}[f]\,|p|^{k}\\
\leq C_{k}&\int^{\infty}_0\int^{\infty}_0 \text{d}|p_{1}|\text{d}|p_{2}|K(|p_{1}|,|p_{2}|)f(|p_{1}|)f(|p_{2}|)\big(|p_{2}|\,|p_{1}|^{k-1} + |p_{2}|^{k-1}|p_{1}|\big)\\
& = 2\,C_{k}\int^{\infty}_0\int^{\infty}_0 \text{d}|p_{1}|\text{d}|p_{2}|K(|p_{1}|,|p_{2}|)f(|p_{1}|)f(|p_{2}|)|p_{2}|\,|p_{1}|^{k-1}\\
&\hspace{1cm} = 4\,C_{k}\,\big(m_{k+3}\langle f\rangle\, m_{3}\langle f\rangle + m_{k+1}\langle f\rangle\, m_{5}\langle f\rangle\big)\,.
\end{align*}
In the last inequality we used that $K(|p_{1}|,|p_{2}|)\leq 2|p_{1}|^{2}|p_{2}|^{2}(|p_{1}|^{2} + |p_{2}|^{2})$.

Regarding the linear part, it follows from a direct computation that
\begin{align*}
&\frac{1}{8\pi^{2}}\int_{\mathbb{R}^{3}}\text{d}p\,\mathcal{L}[f]\,|p|^{k}\\
=\int^{\infty}_0&\int^{\infty}_0 \text{d}|p_{1}|\text{d}|p_{2}|K(|p_{1}|,|p_{2}|)f(|p_{1}|+|p_{2}|)\Big(\big( |p_{1}|+|p_{2}| \big)^{k} - |p_{1}|^{k} - |p_{2}|^{k}\Big)\\
=\int^{\infty}_0&\text{d}|p|\,f(|p|)|p|^{k+6}\int^{|p|}_0\text{d}|p_{1}|\big(\tfrac{|p_{1}|}{|p|}\big)^{2}\big(1-\tfrac{|p_{1}|}{|p|}\big)^{2}\Big(1 - \big(\tfrac{|p_{1}|}{|p|}\big)^{k} - \big(1-\tfrac{|p_{1}|}{|p|}\big)^{k}\Big)\,.
\end{align*}
The result follows after the change of variables $z=|p_{1}|/|p|$ in the inner integral.  
\end{proof}
\begin{theorem}[Propagation of polynomial moments]\label{Theorem:PropaPolynomial}
Let $(f,n_{c}) \geq 0$ be a solution to the problem \eqref{BGP} with finite energy and initial $k^{th}$ moment $m_{k}\langle f_0\rangle<\infty$, for fixed $k>3$.  Then, there exists a constant ${C}_k>0$ that depends only on $k$ such that
\begin{equation}\label{thm4.1.2}
\sup_{t\in[0,T]}m_{k}\langle f\rangle(t)\leq\max\big\{m_{k}\langle f_0 \rangle,C_{k}\,m^{\frac{k+1}{4}}_{3}\big\}\,.
\end{equation}
Here $T>0$ is any time such that $n_{c}(t)>0$ for $t\in[0,T]$.
\end{theorem}
\begin{proof}
Use the weak formulation for $f$ with $\varphi(|p|)=|p|^{k}$, $k>1$.  Then, using Lemma \ref{control1}
\begin{equation*}
\frac{\text{d}}{\text{d}t}m_{k+2}(t)\leq \tfrac{\kappa_0}{n_{c}(t)}\Big(C_{k}\big(m_{k+3}(t)m_{3}(t) + m_{k+1}(t)m_{5}(t)\big) - c_{k}m_{k+7}(t)\Big)\,.
\end{equation*}
Using the interpolations
\begin{equation*}
m_{k+3}\leq m^{\frac{4}{k+4}}_{3}\,m^{\frac{k}{k+4}}_{k+7}\,,\quad m_{k+1}\leq m^{\frac{6}{k+4}}_{3}m^{\frac{k-2}{k+4}}_{k+7}\,,\quad\text{and}\quad m_{5}\leq m^{\frac{k+2}{k+4}}_{3}m^{\frac{2}{k+4}}_{k+7}\,,
\end{equation*}
one concludes that (we drop the time dependence for simplicity)
\begin{equation*}
\frac{\text{d}}{\text{d}t}m_{k+2}\leq \tfrac{\kappa_{0}}{n_{c}(t)}\Big(C_{k}m^{\frac{k+8}{k+4}}_{3}m^{\frac{k}{k+4}}_{k+7} - c_{k}m_{k+7}\Big) \leq \tfrac{\kappa_{0}}{n_{c}(t)}\Big(C'_{k}m^{\frac{k+8}{4}}_{3} - \tfrac{c_{k}}{2}m_{k+7}\Big)\,.
\end{equation*}
Now, interpolating again
\begin{equation*}
m_{k+7}\geq m^{-\frac{5}{k-1}}_{3}m^{\frac{k+4}{k-1}}_{k+2}
\end{equation*}
and simplifying, one finally concludes that
\begin{equation}\label{e-ineq-moments}
\frac{\text{d}}{\text{d}t}m_{k+2}\leq \tfrac{\kappa_0}{n_{c}(t)}m^{-\frac{5}{k-1}}_{3}\Big(\tilde{C}_{k}\,m^{\frac{(k+4)(k+3)}{4(k-1)}}_{3} - \tilde{c}_{k}\,m^{\frac{k+4}{k-1}}_{k+2}\Big)\,,
\end{equation}
for some positive constants $\tilde{C}_{k}$ and $\tilde{c}_{k}$ depending only on $k>1$.  The result follows directly from \eqref{e-ineq-moments} after observing that
\begin{equation*}
Y(t):=\max\big\{m_{k+2}(0),\big(\tilde{C}_{k}/\tilde{c}_{k}\big)^{\frac{k-1}{k+4}}m^{\frac{k+3}{4}}_{3}\big\}\,,
\end{equation*}
is a super-solution of \eqref{e-ineq-moments}, thus, $Y(t)\geq m_{k+2}(t)$. 
\end{proof}
\section{$L^{\infty}$-estimate and BEC stability}
In this section we find natural conditions on the initial condition for global existence of solutions.  Although global solutions are not expected to exists for arbitrary $(f_0,n_{0})$, we essentially prove that if $n_{0}>0$ is sufficiently large relatively to the amount of quasi-particles near zero temperature, the BEC will remain formed.
\begin{lemma}\label{quadratic-linear-infinity}
For any suitable $f\geq0$, the quadratic operator can be estimated as
\begin{equation*}
\mathcal{Q}_{q}[f](|p|)\leq 2\,m_{3}\,|p|\,\big\|f(|\cdot|)|\cdot|^{2}\big\|_{L^\infty} - 4\,m_{3}\,|p|\,\big(f(|p|)|p|^{2}\big)\,.
\end{equation*}
In addition, the linear operator satisfies
\begin{equation*}
\mathcal{L}[f](|p|)\leq 2\,m_{4}\,|p|^{2} - c_{0}\,|p|^{5}\,\big(f(|p|)|p|^{2}\big)\,,\quad c_{0}:=\int^{1}_{0}z^{2}(1-z)^{2}\text{d}z\,.
\end{equation*} 
\end{lemma}
\begin{proof}
Recall the strong formulation of $\mathcal{Q}_{q}[f]$ given in Corollary \ref{Lemma:StrongFormulation}
\begin{align}\label{Qnc-bilinear-terms}
&\mathcal{Q}_{q}[f](|p|)=\int^{|p|}_{0}\text{d}|p_{1}| K(|p_{1}|,|p|-|p_{1}|)f(|p_{1}|)f(|p|-|p_{1}|)\nonumber\\
&+ \int^{\infty}_{|p|}\text{d}|p_{1}| \Big(K(|p|,|p_{1}|-|p|)+K(|p_{1}|-|p|,|p|)\Big)f(|p_{1}|)f(|p_{1}|-|p|)\nonumber\\
&+ f(|p|)\bigg(\int^{\infty}_{|p|}\text{d}|p_{1}| K(|p|,|p_{1}|-|p|)f(|p_{1}|) - \int^{\infty}_{0}\text{d}|p_{1}| K(|p|,|p_{1}|)f(|p_{1}|)\nonumber\\
&\qquad\qquad - \int^{|p|}_{0}\text{d}|p_{1}|K(|p|-|p_{1}|,|p_{1}|)f(|p_{1}|)\bigg)\nonumber\\
&+ f(|p|)\bigg(\int^{\infty}_{|p|}\text{d}|p_{1}| K(|p_{1}|-|p|,|p|)f(|p_{1}|) - \int^{\infty}_{0}\text{d}|p_{1}| K(|p_{1}|,|p|)f(|p_{1}|)\nonumber\\
&\qquad - \int^{|p|}_{0}\text{d}|p_{1}| K(|p_{1}|,|p|-|p_{1}|)f(|p_{1}|)\bigg)=:\sum_{i=1}^9  B_i[f](|p|)\,.
\end{align}
For the first term $B_{1}[f](|p|)$ use
\begin{align*}
K(|p_{1}|,|p|-|p_{1}|)&=|p_{1}|^{2}|p|^{2}(|p|-|p_{1}|)^{2} = |p||p_{1}|^{2}(|p| - |p_{1}|)^{2}(|p|- |p_{1}| + |p_{1}|)\\
& =|p||p_{1}|^{3}(|p|-|p_{1}|)^{2} +  |p| |p_{1}|^{2}(|p|-|p_{1}|)^{3}\,.
\end{align*}
For the second term $B_{2}[f](|p|)$, use that in the set $\{|p_{1}|\geq |p|\}$
\begin{align*} 
K(|p|,|p_{1}|-|p|)=|p|^{2}|p_{1}|^{2}(|p_{1}|-|p|)^{2} \leq |p||p_{1}|^{3}(|p_{1}|-|p|)^{2}\,,
\end{align*}
and with an identical estimate for $B_{3}[f](|p|)$.  We obtain, after a change of variables, that
\begin{align*}
B_{1}[f](|p|) + &B_{2}[f](|p|) + B_{3}[f](|p|)\\
&\leq 2|p|\int^{\infty}_{0}|p_{1}|^{3}f(|p_{1}|)\big| |p| - |p_{1}| \big|^{2}f\big(\big| |p|-|p_{1}| \big| \big)\text{d}|p_{1}|\\
&\hspace{2cm}\leq2\,|p|\,\big\|f(\cdot)|\cdot|^{2}\big\|_{L^{\infty}}\,m_{3}\,.
\end{align*}
Now, the sum of the terms $4^{th}, 5^{th}$ and $6^{th}$ can be rewritten as
\begin{align*}
B_{4}[f]&(|p|) + B_{5}[f](|p|) + B_{6}[f](|p|) = \\
&f(|p|)\bigg(\int^{\infty}_{|p|}\big(K(|p|,|p_{1}|-|p|) - K(|p|,|p_{1}|)\big)f(|p_{1}|)\text{d}|p_{1}|\\
&\qquad\qquad - \int^{|p|}_{0}\big(K(|p|,|p_{1}|) + K(|p|-|p_{1}|,|p_{1}|)\big)f(|p_{1}|)\text{d}|p_{1}|\bigg)\,.
\end{align*}
Note that an explicit calculation gives
\begin{align*}
K(|p|,|p_{1}|-|p|) - & K(|p|,|p_{1}|) = -4|p|^{3}|p_{1}|^{3}\,.
\end{align*}
Also, in the set $\{|p_{1}|\leq |p|\}$ it follows
\begin{align*}
&K(|p|,|p_{1}|)+K(|p|-|p_{1}|,|p_{1}|)=2|p|^{2}|p_{1}|^{2}\big(|p|^{2} + |p_{1}|^{2}\big)\geq 2|p|^{3}|p_{1}|^{3}\,.
\end{align*}
Therefore, this sum can be estimated as
\begin{align*}
&B_{4}[f](|p|) + B_{5}[f](|p|) + B_{6}[f](|p|)\\
& \leq -2\,|p|^{3}\,f(|p|)\bigg(2\int^{\infty}_{|p|}|p_{1}|^{3}f(|p_{1}|)\text{d}|p_{1}| + \int^{|p|}_{0}|p_{1}|^{3}f(|p_{1}|)\text{d}|p_{1}|\bigg)\\
& \leq -2\,|p|^{3}\,f(|p|)\,m_{3}\,.
\end{align*}
Now, by symmetry $K(|p|,|p_{1}|)=K(|p_{1}|,|p|)$, one has the identity $B_{4}[f](|p|) + B_{5}[f](|p|) + B_{6}[f](|p|)=B_{7}[f](|p|) + B_{8}[f](|p|) + B_{9}[f](|p|)$, and consequently
\begin{equation*}
\mathcal{Q}_{q}[f](|p|)\leq 2\,m_{3}\,|p|\,\big\|f(|\cdot|)|\cdot|^{2}\big\|_{L^{\infty}} - 4\,m_{3}\,|p|\,\big(f(|p|)|p|^{2}\big)\,.
\end{equation*}
Now, the strong formulation of the linear operator reads
\begin{align}\label{Qnc-linear-terms}
\begin{split}
&\mathcal{L}[f](|p|) = L_1[f](|p|) + L_2[f](|p|) + L_3[f](|p|):=\\
&\qquad\int^{\infty}_{|p|}\big(K(|p|,|p_{1}|-|p|)+K(|p_{1}|-|p|,|p|)\big)f(|p_{1}|)\text{d}|p_{1}| \\
&\qquad\qquad - f(|p|)\int^{|p|}_{0}K(|p_{1}|,|p|-|p_{1}|)\text{d}|p_{1}|\,.
\end{split}
\end{align}
Note that $K(|p|,|p_{1}|-|p|)=|p|^{2}|p_{1}|^{2}(|p_{1}|-|p|)^{2}\leq |p|^{2}|p_{1}|^{4}$ in the set $\{|p|\leq|p_{1}|\}$, thus,
\begin{equation*}
L_1[f](|p|) + L_2[f](|p|) \leq 2\,|p|^{2}\,m_{4}\,.
\end{equation*}
Finally, an elementary calculation gives for the current kernel $K(|p|,|p_{1}|)=|p|^{2}(|p|+|p_{1}|)^{2}|p_{1}|^{2}$
\begin{align*}
L_{3}[f](|p|) &= f(|p|)\int^{|p|}_{0}K(|p_{1}|,|p|-|p_{1}|)\text{d}|p_{1}|\\
& = \int^{1}_{0}z^{2}(1-z)^{2}\text{d}z\,f(|p|)\,|p|^{7}=:c_{0}\,f(|p|)\,|p|^{7}\,.
\end{align*}
\end{proof}
\begin{proposition}[$L^{\infty}$-estimate]\label{estimate-infinity}
Let $(f,n_{c})\geq0$ be a solution of \eqref{BGP} with finite energy and $4^{th}$ moment.   Also, assume that $n_{c}(\cdot)$ is absolutely continuous and that $\|f_{0}(|\cdot|)|\cdot|^{2}\|_{L^{\infty}}<\infty$.  Then,
\begin{equation*}
\sup_{0\leq s \leq T}\big\|f(s,\cdot)|\cdot|^{2}\big\|_{L^{\infty}}\leq \max\bigg\{\big\|f_0(\cdot)|\cdot|^{2}\big\|_{L^{\infty}}, \frac{3\sup_{0\leq s \leq T}m_{4}(s)}{2\,c^{1/4}_{0}m^{3/4}_{3}}\bigg\}\,.
\end{equation*}
Here $T>0$ is any time such that $n_{c}(t)>\delta$ for $t\in[0,T]$ and for some fixed constant $\delta>0$.
\end{proposition}
\begin{proof}
The weak formulation leads to the strong representation
\begin{equation*}
\partial_{t}f(t,|p|)|p|^{2} = \tfrac{\kappa_0}{n_{c}(t)}\Big(\mathcal{Q}_{q}[f(t)](|p|) + \mathcal{L}[f(t)](|p|)\Big)\,,\; t\geq0\,,\, |p|\geq0\,.
\end{equation*}
Since $n_{c}(\cdot)>0$ is absolutely continuous in $[0,T]$, it is possible to solve uniquely the nonlinear ode
\begin{equation}
\alpha'(t) = \frac{1}{n_{c}(\alpha(t))}\,,\qquad \alpha(0)=0\,,
\end{equation}
in the region $0\leq\alpha(t)\leq T$.  The function $\alpha$ is strictly increasing. 

Observe that
$$\int_{\mathbb{R}^3}dp\,f(t,p) \ + \  n_c(t)=\int_{\mathbb{R}^3}dp\, f_0(p) \ + \  n_c(0)=C(f_0,n_c(0));$$
hence, $n_c$ is uniformly bounded in time by $C(f_0,n_c(0))$, then 
$$ \frac{1}{n_{c}(\alpha(t))}\ge \frac{1}{C(f_0,n_c(0))}>0.$$
Thus the function $\alpha$ is strictly increasing and $\lim_{t\to\infty}\alpha(t)=\infty$.  Let $\tilde{T}$ be the unique time such that $\alpha(\tilde{T})=T$ and define the time scaled function
\begin{equation*}
F(t,|p|): = f\big(\alpha(t),|p|\big)\,,\quad t\in[0,\tilde{T}]\,.
\end{equation*}

It follows that
\begin{equation*}
\partial_{t}F(t,|p|)|p|^{2} = \kappa_0\Big(\mathcal{Q}_{q}[F(t)](|p|) + \mathcal{L}[F(t)](|p|)\Big)\,,\quad |p|\geq0\,,
\end{equation*}
valid in the interval $t\in[0,\tilde{T}]$.
Clearly, $m_{3}\langle F(t)\rangle =m_{3}\langle F(0) \rangle = m_{3}\langle f_0 \rangle=:m_{3}$.  Define for simplicity $g(t,|p|):=F(t,|p|)|p|^{2}$ and use the weak formulation and Lemma \ref{quadratic-linear-infinity} to obtain
\begin{align*}
\partial_{t}g(t,&|p|)\leq 2\, m_{3}\, |p|\,\|g(t,|p|)\|_{\infty}  \\
&- 4\,m_{3}\,|p|\,g(t,|p|) + 2\,m_{4}\,|p|^{2} - c_{0}\,|p|^{5}\,g(t,|p|)\,.
\end{align*}
Integrating this differential inequality,
\begin{align*}
g(t,&|p|)\leq g(0,|p|)e^{-|p|(4m_{3}+c_0|p|^{4})t}\\
&\quad+2|p|\int^{t}_{0}e^{-|p|(4m_{3}+c_{0}|p|^{4})(t-s)}\Big(m_{3}\big\|g(s,\cdot)\big\|_{L^{\infty}}+m_{4}(s)|p|\Big)\text{d}s\\
&\leq \max\bigg\{\|g(0,|p|)\|_{\infty}\,,\frac{2\,m_{3}\sup_{s}\big\|g(s,\cdot)\big\|_{L^{\infty}}+2\sup_{s}m_{4}(s)\,|p|}{4m_{3}+c_{0}\, |p|^{4}}\bigg\}\\
&\leq \max\bigg\{\|g(0,|p|)\|_{\infty}\,, \tfrac{1}{2}\sup_{s}\big\|g(s,\cdot)\big\|_{L^{\infty}} + \Big\|\tfrac{2|\cdot|}{4m_{3}+c_{0}\, |\cdot|^{4}}\Big\|_{\infty}\sup_{s}m_{4}(s)\bigg\}\,.
\end{align*}
All supremum are taken in $s\in[0,\tilde{T}]$.  Since
\begin{equation*}
\frac{2|p|}{4m_{3}+c_{0}\, |p|^{4}}\leq \frac{3^{3/4}}{2^{5/2}c^{1/4}_{0}m^{3/4}_{3}}\,,
\end{equation*}
it follows, after taking supremum in $|p|\geq0$ and then in $t\geq0$, that
\begin{equation*}
\sup_{s}\big\|g(s,\cdot)\big\|_{L^{\infty}} \leq \max\bigg\{\big\|g(0,\cdot)\big\|_{L^{\infty}}, \frac{3^{3/4}\sup_{s}m_{4}(s)}{2^{3/2}\,c^{1/4}_{0}m^{3/4}_{3}}\bigg\}\,.
\end{equation*}
The result follows since
\begin{align*}
\sup_{s\in[0,\tilde{T}]}\big\|g(s,\cdot)\big\|_{L^{\infty}} &= \sup_{s\in[0,T]}\big\|f(s,\cdot)|\cdot|^{2}\big\|_{L^{\infty}}\,,\quad \text{and}\\
\sup_{s\in[0,\tilde{T}]}m_{4}\langle F(s) \rangle &= \sup_{s\in[0,T]}m_{4}\langle f(s) \rangle\,.
\end{align*}
\end{proof}
Observe that for any $\varepsilon>0$,
\begin{align}\label{controlm2}
\begin{split}
m_{2}\langle f \rangle=\int^{\infty}_{0}&\text{d}|p|f(|p|)|p|^{2} = \int^{\varepsilon}_{0}\text{d}|p|f(|p|)|p|^{2} +\int^{\infty}_{\varepsilon}\text{d}|p|f(|p|)|p|^{2}\\
&\leq \varepsilon\big\|f(\cdot)|\cdot|^{2}\big\|_{L^{\infty}} + \tfrac{1}{\varepsilon}m_{3} \leq 2\sqrt{m_{3}\,\big\|f(|\cdot|)|\cdot|^{2}\big\|_{L^{\infty}}}\,,
\end{split}
\end{align}
where the last inequality follows after minimization over $\varepsilon>0$.  This simple observation proves most of the following theorem.
\begin{theorem}[BEC stability]\label{BECstab}
Let $(f,n_{c})\geq0$ be a solution of \eqref{BGP} with finite energy and $4^{th}$ moment.   Also, assume that $n_{c}(\cdot)>0$ is absolutely continuous and that $\|f_{0}(|\cdot|)|\cdot|^{2}\|_{L^{\infty}}<\infty$.  Then, there exists a threshold $\mathcal{C}(f_0)>0$, that can be taken as in \eqref{th}, such that for any initial BEC having mass
\begin{equation}\label{ncthres}
n_{c}(0) \geq \mathcal{C}(f_0) - m_{2}(0)+\delta\,,\qquad \delta>0\,,
\end{equation}
then, the BEC remains uniformly formed, $$\inf_{0\leq s \leq T} n_{c}(s)\geq \delta\,.$$  Here $T>0$ is any time were the aforementioned assumptions hold.
\end{theorem}
\begin{proof}
For a solution $(f(t,\cdot),n_{c}(t))\geq 0$ of \eqref{BGP} with continuous moments and with $n_{c}(t)>0$ in $[0,T]$, the pair $(f(t), n_{c}(t))$ enjoys the total conservation of mass $m_{2}(t) + n_{c}(t)=m_{2}(0) + n_{c}(0)$ in such interval.  Then, using \eqref{controlm2}
\begin{align*}
n_{c}(t)&=n_{c}(0)+m_{2}(0) - m_{2}(t)\geq n_{c}(0)+m_{2}(0) - \sup_{0\leq s \leq T}m_{2}(s)\\
&\geq n_{c}(0)+m_{2}(0) - 2\sqrt{m_{3}\,\sup_{0\leq s \leq T}\big\|f(s,|\cdot|)|\cdot|^{2}\big\|_{L^{\infty}}}\,.
\end{align*}
Moreover, using Proposition \ref{estimate-infinity} and Theorem \ref{Theorem:PropaPolynomial},
\begin{align}
&2\sqrt{m_{3}\,\sup_{0\leq s \leq T}\big\|f(s,|\cdot|)|\cdot|^{2}\big\|_{L^{\infty}}}\nonumber\\
&\quad\leq 2\sqrt{m_{3}\max\bigg\{\big\|f_0(\cdot)|\cdot|^{2}\big\|_{L^{\infty}}, \frac{3\sup_{0\leq s\leq T}m_{4}(s)}{2\,c^{1/4}_{0}m^{3/4}_{3}}\bigg\}}\nonumber\\
&\leq 2\sqrt{m_{3}\max\bigg\{\big\|f_0(\cdot)|\cdot|^{2}\big\|_{L^{\infty}}, \frac{3\max\big\{m_{4}\langle f_0 \rangle,C\,m^{\frac{5}{4}}_{3}\big\}}{2\,c^{1/4}_{0}m^{3/4}_{3}}\bigg\}} =:\mathcal{C}(f_0)\,.\label{th}
\end{align}
Thus, fixing $\delta>0$, if
\begin{equation*}
n_{c}(0)\geq\mathcal{C}(f_0) - m_{2}(0)+\delta\,,
\end{equation*}
we have $\inf_{0\leq s \leq T}n_{c}(s)\geq\delta$ which concludes the proof.
\end{proof}

\section{The Cauchy Problem}
This section is devoted to show existence and uniqueness of positive solutions of the initial value problem \eqref{BGP} with quantum interaction operator $\mathcal{Q}[f]$ defined in \eqref{GainLossSplit}, \eqref{GainOperator} and \eqref{LossOperator}, associated to a transition probability $|\mathcal M|^2=\kappa|p||p_1||p_2|$ valid in the low temperature regime.

The first observation is that the system \eqref{BGP} can be reduced to a single equation after explicit integration of $n_{c}(t)$.  Indeed,
\begin{equation}\label{nc}
n_{c}[f](t):=n_{c}(t)=\sqrt{n^{2}_0 - 2\kappa_0\int^{t}_{0}\text{d}s\int_{\mathbb{R}^{3}}\text{d}p\,\mathcal{Q}[f](s,p)}\,.
\end{equation}
As a consequence, system \eqref{BGP} is equivalent to the single equation
\begin{equation}\label{EBGP}
\frac{\text{d}f}{\text{d}t} = \tfrac{\kappa_0}{n_{c}[f]}\mathcal{Q}[f]\,,\quad t>0\,,
\end{equation}
complemented with the initial condition $f(0,\cdot)=f_{0}(\cdot)$.  This equivalence is valid as long as $n_{c}(\cdot)>0$.  Note that equation \eqref{EBGP} is an nonlinear equation with memory.

The approach we follow here is based on an abstract ODE framework in Banach spaces.  The following theorem, proved in the Appendix \ref{Appendix}, is valid for causal operators.  Fix spaces $\mathcal{S}$ and $E$, time $T>0$, and causal operator 
\begin{equation}
\mathcal{O}:\mathcal{C}\big([0,T];\mathcal{S}\big)\rightarrow \mathcal{C}\big([0,T];E\big).
\end{equation}  We recall that an operator $\mathcal{O}$ is causal, if for any $t\in[0,T]$ the operator at time $t$ is defined only by the values of $f$ in $[0,t]$, that is, $\mathcal{O}[f](t) = \mathcal{O}[f(\cdot)\textbf{1}_{\{\cdot\leq t\}}](t)$.

\begin{theorem}\label{Theorem_ODE} Let $E:=(E,\|\cdot\|)$ be a Banach space, $\mathcal{S}$  be a bounded, convex and closed subset of $E$, and $\mathcal{O}:\mathcal{C}\big([0,T];\mathcal{S}\big)\rightarrow \mathcal{C}\big([0,T];E\big)$ be a causal operator  satisfying the following properties:
\begin{itemize}
\item [$\cdot$] H\"{o}lder continuity condition: For any functions $f,g\in\mathcal{C}\big([0,T];\mathcal{S}\big)$ and times $0 \leq t \leq s\in[0,T]$, there is $\beta\in(0,1)$ such that
\begin{align}\label{Holder_C}
\begin{split}
\big\|&\mathcal{O}[f](t) - \mathcal{O}[g](s)\big\|\\
&\leq C\big(\sup_{\sigma\in[0,t]}\big\|f(\sigma) - g(\sigma)\big\|^{\beta} + \big\|f(t) - g(s)\big\|^{\beta} + |t-s|^{\beta}\big)\,,
\end{split}
\end{align}
\item [$\cdot$] sub-tangent condition: For any $f\in\mathcal{C}\big([0,T];\mathcal{S}\big)$
\begin{equation}\label{subtangent}
\liminf_{h\rightarrow0^+}\,h^{-1}\sup_{t\in[0,T]}\text{dist}\big(f(t)+h\,\mathcal{O}[f](t),\,\mathcal{S}\big)=0\,,
\end{equation}
\item [$\cdot$] and, one-sided Lipschitz condition:  For any $f,g\in\mathcal{C}\big([0,T];\mathcal{S}\big)$ and $t\in[0,T]$
\begin{equation}\label{one-sided-lip}
\int^{t}_{0}ds\big[ \mathcal{O}[f](s) - \mathcal{O}[g](s), f(s) - g(s) \big] \leq L\int^{t}_{0}ds\big\|f(s) - g(s)\big\|\,,
\end{equation}
where $\big[ \varphi,\phi \big]: = \lim_{h\rightarrow 0^{-}}h^{-1}\big(\| \phi + h\varphi \| - \| \phi \| \big)$.
\end{itemize}
Then, the equation 
\begin{equation}\label{Theorem_ODE_Eq}
\partial_t f=\mathcal{O}[f] \mbox{ on } [0,T)\times E,~~~~f(0)=f_0 \in \mathcal{S}
\end{equation}
has a unique solution in $\mathcal{C}^1\big([0,T);E\big)\cap \mathcal{C}\big([0,T);\mathcal{S}\big)$.
\end{theorem}

\noindent
This theorem is an extension of Theorem A.1 proved in \cite{Bressan} by Bressan in the context of solving the elastic Boltzmann equation for hard spheres in $3$ dimension. We point out that \cite{Bressan} does not properly show  that \eqref{subtangent} is satisfied in that case. For completeness of this manuscript we rewrite Bressan's unpublished proof in the Appendix.
The Bressan's needed techiques  can be found in \cite{Martin}. Indeed, referring to the argument given in \cite{ABCL:2016}, using conditions \eqref{Holder_C} and  \eqref{subtangent} combined with  \cite[Theorem VI.2.2]{Martin} one has that conditions (C1), (C2) and (C3) in \cite[pg. 229]{Martin} are satisfied and  
hence, together with \eqref{one-sided-lip}, all needed conditions  for the existence and uniqueness theorem  \cite[Theorem VI.4.3]{Martin}  for ODEs in Banach spaces are fulfilled.  

For our  particular case, we need to identify a suitable Banach space and a corresponding  bounded, convex and closed subset $\mathcal{S}$.  Choosing $E=L^{1}\big(\mathbb{R}^{3},\text{d}p\big)$ as Banach space, the choice of the subspace $\mathcal{S}$, defined below in \eqref{SetS}, depends on the \textit{a priori} estimates discussed in previous two sections and the desired continuity properties needed for existence. 
 
\smallskip

More specifically, such subset $\mathcal{S}\subset L^{1}\big(\mathbb{R}^{3},\text{d}p\big) $ is characterized by the H\"{o}lder continuity and sub-tangent conditions \eqref{Holder_C} and \eqref{subtangent}, respectively, (to be shown next in subsection~\ref{Subtangent}), and it is defined as follows:
\begin{align}\label{SetS} 
\begin{split}
\mathcal{S} \ :=\ \bigg\{ f \in L^{1}\big(\mathbb{R}^{3},\text{d}&p\big)\;\big| \; {\bf i.}\ f \text{\;nonnegative \& radially symmetric}\,,\\
&{\bf ii.}\  m_3\langle f\rangle \ = \ \int_{\mathbb{R}_+}d|p|\,f(|p|)|p|^3 \ = \ \mathfrak{h}_3\, ,\\
&{\bf iii.}\ m_{8}\langle f \rangle \ =\ \int_{\mathbb{R}_+}d|p|\,f(|p|)|p|^{8}\ \leq\ \mathfrak{h}_{8}\, , \\
&{\bf iv.}\ \| f(\cdot)\,|\cdot|^{2}\, \|_{\infty} \leq \mathfrak{h}_{\infty}<\infty \bigg\},
\end{split}
\end{align}
where $\mathfrak{h}_3>0$ is an arbitrary initial energy.  The specific $\mathfrak{h}_{8}>0$ is defined below in \eqref{h10_0}, and $\mathfrak{h}_{\infty}>0$ will be taken sufficiently large depending only on $\mathfrak{h}_3$ and $\mathfrak{h}_8$.  We are now in conditions to state and prove the global well-posedness theorem.
\begin{theorem}[Global well-posedness]\label{Theorem:ExistenceUniqueness}
Let $f_0(p)=f_0(|p|) \in \mathcal{S}$ and assume that $(f_0,n_0=n_{c}(0))$ satisfies the threshold condition \eqref{ncthres} for $\delta>0$.  Then, system \eqref{BGP} (equivalently, system \eqref{nc}-\eqref{EBGP}) has a unique conservative solution $(f,n_{c})$ such that
\begin{align}\label{the_theorem}
\begin{split}
0&\leq f(t,p)=f(t,|p|)\in \mathcal{C}\big([0,T]; \mathcal{S}\big)\cap  \mathcal{C}^1\big((0,T];L^1\big(\mathbb{R}^{3},dp\big)\big)\,,\\
\delta&\leq n_{c}(t)=n_{c}[f](t) \in \mathcal{C}\big([0,T]\big)\cap\mathcal{C}^1\big((0,T]\big)\,,
\end{split}
\end{align}
for any $T>0$.  Momentum and energy are conserved for $f(t,\cdot)$, and the total mass of the system is conserved as well $$m_{2}\langle f(t) \rangle + n_{c}[f](t) = m_{2}\langle f_0 \rangle + n_0\,.$$ 
\end{theorem}
\begin{proof}
The proof of this theorem consists of verifying the three conditions \eqref{Holder_C}, \eqref{subtangent}, and \eqref{one-sided-lip} to apply Theorem \ref{Theorem_ODE}, respectively for the nonlinear causal operator $\mathcal{O}[f]=\tfrac{\kappa_0}{n_{c}[f]}\mathcal{Q}[f]$.

In the following estimates we fix a time $T:=T_{\delta}>0$ such that $$\inf_{0\leq s \leq T_{\delta}}n_{c}[f](s)\geq\delta\,.$$  This can be done in the space $\mathcal{C}\big([0,T];\mathcal{S}\big)$ since $$\int_{\mathbb{R}^{3}}\text{d}p\big|\mathcal{Q}[f(t)]\big|\leq C\big(m_{2}\langle f(t) \rangle, m_{7}\langle f(t) \rangle \big) \leq C(\mathfrak{h}_{3}, \mathfrak{h}_{8}, \mathfrak{h}_{\infty})\,.$$  In the sequel, we write $C(\mathcal{S})$ for a constant depending only on the parameters defining the set $\mathcal{S}$, namely $\mathfrak{h}_{3},\, \mathfrak{h}_{8}$, and $\mathfrak{h}_{\infty}$.   Therefore, from the definition of $n_{c}[f]$ it suffices to take $T_{\delta}:=\tfrac{n^{2}_{0} - \delta^{2}}{2\kappa_0C(\mathcal{S})}>0$ to satisfy such lower bound on the condensate mass.  A posteriori, knowing the total conservation of mass, we use Theorem \ref{BECstab} to conclude that $T>0$ is, in fact, arbitrary.   

\subsection{H\"{o}lder Estimate}\label{Section:Lipschitz}

Recall the definition of $m_{k}\big\langle f \rangle$, the $k^{th}$-line-moment of a radially symmetric $f(p):=f(|p|)$
\begin{equation}\label{mkft}
m_{k}\langle f \rangle:=\int_{\mathbb{R}_{+}}\text{d}p\,f(|p|)|p|^{k}\,,\qquad k\geq0\,, 
\end{equation}
and observe that $m_{2}\langle |f| \rangle$ is equivalent to the usual norm for a radially symmetric functions in $L^{1}\big(\mathbb{R}^{3},\text{d}p\big)$.
\begin{lemma}[H\"{o}lder continuity]\label{Lemma:Lipschitz}
The collision operator 
$$\tfrac{\kappa_0}{n_{c}[\cdot]}\mathcal{Q}[\cdot]:\mathcal{C}\big([0,T];\mathcal{S}\big) \rightarrow  \mathcal{C}\big([0,T];L^{1}\big(\mathbb{R}^{3},dp\big)\big)$$
is H\"{o}lder continuous with estimate
\begin{align}\label{Lemma:Lipschitz:Eq}
\begin{split}
&m_{2}\big\langle \Big|\tfrac{\kappa_0}{n_{c}[f](t)}\mathcal{Q}[f(t)] - \tfrac{\kappa_0}{n_{c}[g](s)}\mathcal{Q}[g(s)] \Big| \big\rangle\\
&\leq C_{\delta,T}(\mathcal{S})\Big(\sup_{\sigma\in[0,t]}m_{2}\big\langle |f(\sigma)-g(\sigma)| \big\rangle^{\frac{1}{6}} + \sup_{\sigma\in[0,t]}m_{2}\big\langle |f(\sigma)-g(\sigma)|\big\rangle\Big)\\
&\quad+ C_{\delta}(\mathcal{S})\Big(m_{2}\big\langle |f(t)-g(s)| \big\rangle^{\frac{1}{6}} + m_{2}\big\langle |f(t)-g(s)|\big\rangle + |t-s|\Big)\,,
\end{split}
\end{align}
valid for all $f,\,g\in\mathcal{C}\big([0,T];\mathcal{S}\big)$ and $0\leq t \leq s\in[0,T]$.
\end{lemma}
\begin{proof}
Recall that the interaction operator can be written as a sum of a nonlinear part and a linear part $\mathcal{Q}[f] = \mathcal{Q}_{q}[f] + \mathcal{L}[f]$.  Besides, the nonlinear part is the sum of nine terms $\mathcal{Q}_{q}[f] = \sum^{9}_{i=1} B_{i}[f]$, as in \eqref{Qnc-bilinear-terms}, and the linear part is the sum of three terms $\mathcal{L}[f]=\sum^{3}_{i=1}L_{i}[f]$, as in \eqref{Qnc-linear-terms}.   An elementary calculation shows that the nonlinear terms satisfy for $1\leq i \leq 9$
\begin{align*}
&\int_{\mathbb{R}^{3}}\text{d}p\big|B_{i}[f] -B_{i}[g]\big| \\
&\leq 2\max\big\{m_{2}\langle f \rangle, m_{4}\langle f \rangle,m_{2}\langle g \rangle, m_{4}\langle g \rangle\big\}\Big(m_{2}\langle |f-g| \rangle + m_{4}\langle |f - g|\rangle\Big)\\
&\;\leq 2\max\big\{m_{2}\langle f \rangle, m_{4}\langle f \rangle,m_{2}\langle g \rangle, m_{4}\langle g \rangle\big\}\times\\
&\times\Big(m_{2}\langle |f-g| \rangle + \big(m_{8}\langle f \rangle + m_{8}\langle g \rangle\big)^{1/3}m_{2}^{2/3}\langle |f - g|\rangle\Big)\,.
\end{align*}
As for the linear terms,
\begin{equation*}
\int_{\mathbb{R}^{3}}\text{d}p\big|L_{i}[f] -L_{i}[g]\big| \leq m_{7}\langle |f - g|\rangle \leq \big(m_{8}\langle f \rangle + m_{8}\langle g \rangle\big)^{5/6}m_{2}^{1/6}\langle |f - g|\rangle\,.
\end{equation*}
The conclusion is that
\begin{equation}\label{holderQ}
\int_{\mathbb{R}^{3}}\text{d}p\big|\mathcal{Q}[f] - \mathcal{Q}[g]\big| \leq C(\mathcal{S})\Big( m_{2}\langle |f-g| \rangle + m^{1/6}_{2}\langle |f-g| \rangle \Big)\,.
\end{equation} 
Additionally, for any $0\leq t\leq s\in[0,T]$
\begin{align}
&\Big|\frac{1}{n_{c}[f](t)} - \frac{1}{n_{c}[g](s)}\Big| = \frac{\big|n^{2}_{c}[f](t)-n^{2}_{c}[g](s)\big|}{(n_{c}[f](t) + n_{c}[g](s))\,n_{c}[f](t)\,n_{c}[g](s)}\nonumber\\
&\quad\leq 2\kappa_0\,\frac{\int^{t}_{0}\text{d}\sigma\,m_{2}\langle \big|\mathcal{Q}[f(\sigma)] - \mathcal{Q}[g(\sigma)] \big| \rangle + \int^{s}_{t}\text{d}\sigma\,m_{2}\langle \big|\mathcal{Q}[g(\sigma)] \big| \rangle }{(n_{c}[f](t) + n_{c}[g](s))\,n_{c}[f](t)\,n_{c}[g](s)}\nonumber\\
&\leq\frac{C(\mathcal{S})}{\delta^{3}}\Big(\int^{t}_{0}\text{d}\sigma\,m_{2}\big\langle \big|f(\sigma) - g(\sigma)\big| \big\rangle + m^{1/6}_{2}\big\langle \big|f(\sigma) - g(\sigma)\big| \big\rangle + |t-s|\Big)\,.\label{holdernc}
\end{align}
We used, in the last inequality, the fact that $\min\{n_{c}[f],n_{c}[g]\}\geq\delta$ for any $f,g\in\mathcal{C}\big([0,T];\mathcal{S}\big)$.  The result follows after applying $m_{2}\langle \cdot \rangle$ to
\begin{align*}
\begin{split}
\Big|\tfrac{\kappa_0}{n_{c}[f](t)}&\mathcal{Q}[f(t)] - \tfrac{\kappa_0}{n_{c}[g](s)}\mathcal{Q}[g(s)]\Big| \leq \kappa_0\Big| \tfrac{1}{n_{c}[f](t)}-\tfrac{1}{n_{c}[g](s)}\Big| \mathcal{Q}[f(t)]\\
& + \tfrac{\kappa_0}{n_{c}[g](s)}\Big| \mathcal{Q}[f(t)] - \mathcal{Q}[g(s)]\Big|\,.
\end{split}
\end{align*}
and using \eqref{holderQ} and \eqref{holdernc} to estimate each term in the right side.
\end{proof}
\subsection{Sub-tangent condition}\label{Subtangent}
This condition characterizes the stability of the space $\mathcal{S}$ defined in \eqref{SetS} under the equation's dynamics.  Recall that the collision operator $\mathcal{Q}[\cdot]$ can be split as the sum of a gain and a loss operators, as mentioned earlier in \eqref{GainLossSplit} 
\begin{equation*}
\mathcal{Q}[f]=\mathcal{Q}^{+}[f]-f\,\nu[f]\, ,
\end{equation*}
with (refer to the strong formulation and recall the symmetry of $K(\cdot,\cdot)$)
\begin{align}\label{LossOperator3}
\begin{split}
\nu[f](p)&=2\int^{\infty}_{0}\text{d}|p_1|\,K(|p_{1}|,|p|)\,f(|p_1|) + 2\int^{|p|}_{0}\text{d}|p_{1}|\,K(|p_{1}|,|p| - |p_{1}|)\, f(|p_1|)\\
&\hspace{2cm}+\int^{|p|}_{0}\text{d}|p_{1}|\,K(|p_{1}|,|p| - |p_{1}|)\\
&\leq 4|p|^{4} m_2\langle f \rangle + 4|p|^{2}m_4\langle f \rangle + 4|p|^{7} \leq C(\mathcal{S})|p|^{2}(1+|p|^{5})\,.
\end{split}
\end{align}
The sub-tangent condition \eqref{subtangent} follows as a corollary of next Proposition~\ref{prop_subtan}.
\begin{proposition}\label{prop_subtan}
Fix $f\in\mathcal{C}\big([0,T];\mathcal{S}\big)$.  Then, for any $t>0$ and $\epsilon>0$, there exists $h_{*}:=h_{*}(f,\epsilon)>0$, such that the ball centered at $f(t)+h\tfrac{\kappa_0}{n_{c}[f](t)}\mathcal{Q}[f(t)]$ with radius $h\,\epsilon>0$ intersects $\mathcal{S}$, that is,
$$B\Big(f(t)+h\tfrac{\kappa_0}{n_{c}[f](t)}\mathcal{Q}[f(t)],h\,\epsilon\Big)\cap\mathcal{S},\;  \text{is non-empty for any}\; 0<h<h_{*}.$$ 
\end{proposition} 
\begin{proof}  Set $\chi_R(p)$ the characteristic function of the ball of radius $R>0$ and introduce the truncated function $f_R(t,p):=\textbf{1}_{\{|p|\leq R\}}f(t,p)$, then set $w_R(t,p):=f(t,p)+h\tfrac{\kappa_0}{n_{c}[f](t)}\mathcal{Q}[f_R(t)](p)$.

Since $0 \leq f_{R}(t,p) \leq f(t,p)$, one has that
\begin{equation*}
m_{2}\langle f_{R}(t)\rangle \leq m_{2}\langle f(t)\rangle\,, \quad m_{7}\langle f_{R}(t)\rangle \leq m_{7}\langle f(t)\rangle\,.
\end{equation*}
Then, $\tfrac{\kappa_0}{n_{c}[f](t)}\mathcal{Q}[f_{R}(t)]\in \mathcal{C}\big([0,T],L^{1}\big(\mathbb{R}^{3},\text{d}p\big)$ by Lemma~\ref{Lemma:Lipschitz}.  As a consequence, $w_{R}\in \mathcal{C}\big( [0,T] ; L^{1}(\mathbb{R}^{3},\text{d}p)\big)$.
Note that, since $\mathcal{Q}^{+}$ is a positive operator, for any $f(t)\in\mathcal{S}$
\begin{align}\label{positiveS}
\begin{split}
w_{R}(t) & =   f(t) + h\tfrac{\kappa_0}{n_{c}[f](t)}\Big(\mathcal{Q}^{+}[f_{R}(t)] - f_{R}(t)\, \nu[f_{R}(t)]\Big)\\
&\geq f(t) - h\tfrac{\kappa_0}{n_{c}[f](t)}f_{R}(t)\, \nu[f_{R}(t)] \\
&\geq \ f(t)\Big(1 - h\,\delta^{-1}\,C(\mathcal{S})R^{2}\big(1 + R^{5}\big) \Big) \geq 0
\end{split}
\end{align}
for any $0<h< \delta/C(\mathcal{S})R^{2}\big(1 + R^{5}\big)$.  Moreover, by conservation of energy $ \int_{\mathbb{R}_+}\text{d}|p|\,\mathcal{Q}[f_{R}(t)]|p|^{3} =0$ , yielding
\begin{equation}\label{m3R}
\begin{split}
&m_{3}\langle w_{R}(t)\rangle = \int_{\mathbb{R}_+}\text{d}|p|\,w_{R}(t,|p|)|p|^{3}\\
&= \int_{\mathbb{R}_+}\text{d}|p|\,\Big(f(t,|p|)+h\tfrac{\kappa_0}{n_{c}[f](t)}\mathcal{Q}[f_{R}(t)]\Big)|p|^{3}=\int_{\mathbb{R}_+}\text{d}|p|\,f(t,|p|)|p|^{3} = \mathfrak{h}_{3}\,.
\end{split}\end{equation}
In summary, $w_{R}$ satisfies, properties {\bf i.} and {\bf ii.} in the characterization of the $\mathcal{S}$.
Let us show that $w_R$ also satisfies property {\bf iii.} in the set  $\mathcal{S}$.
First, recall the \textit{a priori} estimate \eqref{e-ineq-moments} for the line-moment inequalities, namely
 \begin{align}\label{Section:ProofTheorem:E1}
\begin{split}
\int_{\mathbb{R}^3}\text{d}p\,&\tfrac{\kappa_0}{n_{c}[f](t)}\mathcal{Q}[f(t)]|p|^k \leq \mathcal{L}_{k}\big(t,m_{k}\langle f(t)\rangle \big)\\
:&= \tfrac{\kappa_0}{n_{c}[f](t)}\,m^{-\frac{5}{k-3}}_{3}\Big(\tilde{C}_{k}\,m^{\frac{(k+2)(k+1)}{4(k-3)}}_{3} - \tilde{c}_{k}\,m_{k}\langle f(t)\rangle^{\frac{k+2}{k-3}}\Big)\\
&= \tfrac{\kappa_0}{n_{c}[f](t)}\,\mathfrak{h}^{-\frac{5}{k-3}}_{3}\Big(\tilde{C}_{k}\,\mathfrak{h}^{\frac{(k+2)(k+1)}{4(k-3)}}_{3} - \tilde{c}_{k}\,m_{k}\langle f(t)\rangle^{\frac{k+2}{k-3}}\Big)\,.
\end{split}
\end{align}
This estimate holds for any $k>3$ and $\tilde{C}_k, \tilde{c}_{k}$ only depending on $k$.  Note that the map $\mathcal{L}_{k}(t,\cdot):[0,\infty)\rightarrow \mathbb{R}$ is decreasing and has only one root $\mathfrak{h}^{k}_{*}:=\frac{\tilde{C}_{k}}{\tilde{c}_{k}}\mathfrak{h}^{(k+1)/4}_{3}$, at which $\mathcal{L}_{k}$ changes from positive to negative for any $k>3$.  Note that this root only depends on $\mathfrak{h}_{3}$ and $k$, in particular, it is time independent.  Thus, it is always the case that for any $f\in\mathcal{C}\big([0,T];\mathcal{S}\big)$
\begin{align*}
\int_{\mathbb{R}^3}\text{d}p\,\tfrac{\kappa_0}{n_{c}[f](t)}\mathcal{Q}[f]|p|^k \leq \mathcal{L}_{k}\big(t,m_{k}\langle f\rangle\big) \leq \mathcal{L}_{k}(t,0)\leq \tfrac{\kappa_0}{\delta}\,\tilde{C}_{k}\,\mathfrak{h}^{\frac{(k+6)}{4}}_{3}\,.
\end{align*}
Fix $k=8$ and define
\begin{equation}\label{h10_0}
\mathfrak{h}_{8}:=  2\mathfrak{h}^{8}_{*}+\tfrac{\kappa_0}{\delta}\,\tilde{C}_{8}\,\mathfrak{h}^{\frac{7}{2}}_{3}\,.
\end{equation}
For any $f\in\mathcal{C}\big([0,T];\mathcal{S}\big)$, we have two sets: $I_{1}=\{t:m_{8}\langle f(t)\rangle\leq2\mathfrak{h}^{8}_{*}\}$ and $I_{2}=\{t:m_{8}\langle f(t)\rangle>2\mathfrak{h}^{8}_{*}\}$.  For the former, it readily follows that
\begin{align*}
m_{8}\langle w_{R}(t)\rangle=\int_{\mathbb{R}^{3}}\text{d}p\,&w_{R}(t,|p|)|p|^{8} = \int_{\mathbb{R}^{3}}\text{d}p\,\Big(f(t)+h\tfrac{\kappa_0}{n_{c}[f](t)}\mathcal{Q}[f_R(t)]\Big)|p|^{8}\\
&\leq  2\mathfrak{h}^{8}_{*} + h\tfrac{\kappa_0}{\delta}\,\tilde{C}_{8}\,\mathfrak{h}^{\frac{7}{2}}_{3} \leq \mathfrak{h}_{8},
 \end{align*} 
where in the last inequality we have assumed $h\leq1$ without loss of generality.

For the latter, we can choose $R=R_{1}(f)$ sufficiently large such that $\inf_{t\in I_{2}} m_{8}\langle f_{R}(t)\rangle\geq\mathfrak{h}^{8}_{*}$,  and therefore,
\begin{equation*}
\int_{\mathbb{R}^3}\text{d}p\,\tfrac{\kappa_0}{n_{c}[f](t)}\mathcal{Q}[f_R(t)]|p|^{8} \leq \mathcal{L}_{8}\big(t,m_{8}\langle f_{R}(t) \rangle\big)\leq 0\,,\qquad t\in I_{2}\,.
\end{equation*}
As a consequence, for any $t\in I_{2}$
\begin{align*}
m_{8}\langle w_{R}(t)\rangle = \int_{\mathbb{R}^{3}}&\text{d}p\,\Big(f(t)+h\tfrac{\kappa_0}{n_{c}[f](t)}\mathcal{Q}[f_R(t)]\Big)|p|^{8}\\
&\leq \int_{\mathbb{R}^{3}}\text{d}p\,f(t)|p|^{8}\leq \mathfrak{h}_{8}\,. 
\end{align*}
The conclusion is that for any $f\in\mathcal{C}\big([0,T];\mathcal{S}\big)$, it is always the case that
\begin{equation}\label{propii.}
m_{8}\langle w_{R}(t) \rangle \leq \mathfrak{h}_{8}\,,\quad\text{as long as}\quad R\geq R_{1}(f)>0\,,
\end{equation}
which ensures that $w_{R}$ satisfies property {\bf iii.} of the set $\mathcal{S}$ in \eqref{SetS}.  Let us prove now that $w_{R}$ satisfies property {\bf iv.}  To this end, consider the sets
\begin{align*}
O&=\big\{(t,p) : f(t,p)\,|p|^{2} \geq 0.9\,\mathfrak{h}_{\infty}\big\}\,,\\
O_{R}&=\big\{(t,p) : f_{R}(t,p)\,|p|^{2} \geq 0.9\,\mathfrak{h}_{\infty}\big\}\,.
\end{align*}
In addition, consider the set $W=\big\{(t,p) : w_{R}(t,p)\,|p|^{2} > \mathfrak{h}_{\infty}\big\}$.  Assume that $W$ is of positive measure.  Then,
\begin{align*}
\mathfrak{h}_{\infty} < w_{R}(t,p)&\,|p|^{2} = f(t,p)\,|p|^{2} + h\tfrac{\kappa_0}{n_{c}[f](t)}\mathcal{Q}[f_{R}(t)](p)\,|p|^{2}\\
&\leq f(t,p)\,|p|^{2} + h\tfrac{\kappa_0}{n_{c}[f](t)}\mathcal{Q}^{+}[f_{R}(t)](p)\,|p|^{2}\,,\quad (t,p)\in W\,.  
\end{align*}
It is not difficult to check, using the strong formulation, that for any function $F(t)\in\mathcal{S}$
\begin{equation*}
\|\mathcal{Q}^{+}[F(t)](p)\,|p|^{2}\|_{\infty} \leq 6\,\|F(t,\cdot)\,|\cdot|^{2}\|_{\infty}m_{4}\langle F(t) \rangle\, + 2\,m_{6}\langle F(t) \rangle\leq C(\mathcal{S})\,.
\end{equation*}
Thus,
\begin{align*}
f(t,p)\,|p|^{2} &> \mathfrak{h}_{\infty} - h\tfrac{\kappa_0}{n_{c}[f](t)}\big\|\mathcal{Q}^{+}[f_{R}(t)](p)\,|p|^{2}\big\|_{\infty}\\
&\geq \mathfrak{h}_{\infty} - h\,\delta^{-1}\,C(\mathcal{S})\geq 0.9\,\mathfrak{h}_{\infty}\,,\quad (t,p)\in W\,, 
\end{align*}
where, for the last step, $0 < h\leq 0.1\,\delta\,\mathfrak{h}_{\infty}/C(\mathcal{S})$.  As a consequence, $W\subset O$.  Since $O_{R}\nearrow O$ as $R\rightarrow\infty$, there exists $R=R_{2}(f)>0$ sufficiently large such that $W \cap O_{R}$ is of positive measure. 
Take $(s,q)$ in such intersection, then by Lemma \ref{quadratic-linear-infinity}
\begin{align*}
&w_{R}(s,q)\,|q|^{2} = f(s,q)\,|q|^{2} + h\tfrac{\kappa_0}{n_{c}[f](s)}\mathcal{Q}[f_{R}(s)](q)\,|q|^{2}\\
&\qquad\leq f(s,q)\,|q|^{2} + h\tfrac{\kappa_0}{n_{c}[f](s)}\Big(2\, m_{3}\langle f_{R}(s)\rangle\, |q|\,\big\|f_{R}(s,\cdot)\,|\cdot|^{2}\big\|_{\infty}\\
&- 4\,m_{3}\langle f_{R}(s)\rangle\,|q|\,\big(f_{R}(s,|q|)|q|^{2}\big) + 2\,m_{4}\langle f_{R}(s)\rangle\,|q|^{2} - c_{0}\,|q|^{5}\,\big(f_{R}(s,|q|)|q|^{2}\big)\Big)\\
&\quad \leq \mathfrak{h}_{\infty} + h\tfrac{\kappa_0}{n_{c}[f](s)}|q|\Big( - \tfrac{9}{5}\,m_{3}\langle f_{R}(s)\rangle\,\mathfrak{h}_{\infty} - 0.9\,c_0\,|q|^{4}\,\mathfrak{h}_{\infty} + 2\,m_{4}\langle f_{R}(s)\rangle\,|q|\Big)\,. 
\end{align*}
Using that $m_{4}\leq m^{3/4}_{3}m^{1/4}_{7}$ one obtains that the last parenthesis is majorized by
\begin{align*}
m_{3}\langle f_{R}(s)\rangle\big(- \tfrac{9}{5}\,\mathfrak{h}_{\infty} + \tfrac{3}{2}\big) + \big( \tfrac{1}{2}m_{7}\langle f(s)\rangle - 0.9\,c_0\,\mathfrak{h}_{\infty}\big)|q|^{4}\leq0\,,
\end{align*}
where the non positivity follows by taking $\mathfrak{h}_{\infty}\geq C(\mathfrak{h}_{3},\mathfrak{h}_{8})>0$ sufficiently large.  Therefore, $w_{R}(s,q)\,|q|^{2}\leq \mathfrak{h}_{\infty}$.  This contradicts the definition of $W$, thus, we conclude that $W$ must be empty for this choice of parameters $\mathfrak{h}_{\infty},\, R,$ and $h$.  Then, it is always the case that $\|w_{R}(t,\cdot)\,|\cdot|^{2}\|_{\infty} \leq \mathfrak{h}_{\infty}$ which verifies property {\bf iv}.

We infer due to previous discussion that for any $f\in\mathcal{C}\big([0,T];\mathcal{S}\big)$, there exists $R:=R_{3}(f)$ sufficiently large and $h_{*}:=h_{*}(f,\mathcal{S})>0$ sufficiently small such that that $w_{R}\in\mathcal{C}\big([0,T];\mathcal{S}\big)$ for any $0<h<h_{*}$.

Let us conclude the proof using the H\"older estimate from Lemma \ref{Lemma:Lipschitz} to obtain
\begin{align*}
h^{-1}\,m_{2}\big\langle \big|f(t) &+ h\tfrac{\kappa_0}{n_{c}[f](t)}\mathcal{Q}[f(t)] - w_{R}(t)\big|\big\rangle\\
&=m_{2}\big\langle \big|\tfrac{\kappa_0}{n_{c}[f](t)}\mathcal{Q}[f(t)]-\tfrac{\kappa_0}{n_{c}[f_{R}](t)}\mathcal{Q}[f_R(t)] \big|\big\rangle\\
&\hspace{-2cm}\leq C_{\delta,T}(\mathcal{S})\Big(\sup_{t\in[0,T]}m_{2}\big\langle\big|f(t)-f_{R}(t)\big|\big\rangle^{\frac{1}{6}} + \sup_{t\in[0,T]}m_{2}\big\langle\big| f(t)-f_{R}(t) \big|\big\rangle\Big)\leq \epsilon\,,
\end{align*}
where the last inequality is valid for for $R=R_{4}(f,\epsilon)>0$ sufficiently large. Then, $w_R(t)\in B\Big(f(t)+h\tfrac{\kappa_0}{n_{c}[f](t)}\mathcal{Q}[f(t)],h\,\epsilon\Big)$ for all times provided this choice of $R$.  Thus, choosing $R=\max\big\{R_{3}(f),R_{4}(f,\epsilon)\big\}$, one concludes that 
\begin{equation*}
w_R(t)\in B\Big(f(t)+h\tfrac{\kappa_0}{n_{c}[f](t)}\mathcal{Q}[f(t)],h\,\epsilon\Big)\cap\mathcal{S}\,,\quad 0<h<h_{*},\;\; t\in[0,T]\,.
\end{equation*}
Consequently,
\begin{equation*}
h^{-1}\sup_{t\in[0,T]}\text{dist}\big( f(t) + h\tfrac{\kappa_0}{n_{c}[f](t)}\mathcal{Q}[f(t)] , \mathcal{S}\big) \leq \epsilon\,, \qquad  \forall\,0<h<h_{*}\,. 
\end{equation*}
The proof of  Proposition~\ref{prop_subtan} is now complete and accounts for the sub-tangent condition.
\end{proof}
\subsection{One-side Lipschitz condition}\label{OnesidedLipschitz}
Using dominate convergence theorem one can show that
\begin{equation*}
\big[\varphi(t),\phi(t)\big] \leq \int_{\mathbb{R}^{3}}\text{d}p\,\varphi(t,p)\,\text{sign}(\phi(t,p))\,.
\end{equation*}
Thus, the one-side Lipschitz condition is met after proving the following lemma showing a Lipschitz condition for the interaction Boltzmann operator. The following proof, which yields a uniqueness results, is in the same spirit of the original Di Blassio \cite{DiBlasio} uniqueness proof for initial value problem to the homogeneous Boltzmann equation for hard spheres, using data with enough initial moments.
\begin{lemma}[Lipschitz condition]\label{LipCon}  Assume $f, g \in \mathcal{C}\big([0,T];\mathcal{S}\big)$.  Then, there exists constant $C:=C_{\delta, T}(\mathcal{S})>0$ such that
\begin{align*}
\int^{t}_{0}&ds\int_{\mathbb{R}^{3}}dp\,\Big(\tfrac{\kappa_0}{n_{c}[f](s)}\mathcal{Q}[f(s)] - \tfrac{\kappa_0}{n_{c}[g](s)}\mathcal{Q}[g(s)]\Big)\\
&\times \text{sign}\big(f(s)-g(s)\big)\big( 1 + |p|^{2} \big) \leq C\int^{t}_{0}ds\,m_{2}\big\langle |f(s)-g(s)|\big\rangle\,,\quad t\in[0,T]\,.
\end{align*}
\end{lemma}
\begin{proof}
Writing $\mathcal{Q}[f] = \mathcal{Q}_{q}[f] + \mathcal{L}[f]$, one has that
\begin{align*}
\int_{\mathbb{R}^{3}}&\text{d}p\,\big(\mathcal{Q}[f](p) - \mathcal{Q}[g](p)\big)\big( 1 + |p|^{2} \big)\text{sign}(f-g)=\\
&\int_{\mathbb{R}^{3}}\text{d}p\,\big(\mathcal{Q}_{q}[f](p) - \mathcal{Q}_{q}[g](p)\big)\big( 1 + |p|^{2} \big)\text{sign}(f-g)\\
 + \int_{\mathbb{R}^{3}}&\text{d}p\,\big(\mathcal{L}[f](p) - \mathcal{L}[g](p)\big)\big( 1 + |p|^{2} \big)\text{sign}(f-g)\,.
\end{align*}
For the quadratic part it follows, after a simple inspection of the weak formulation, that
\begin{align}\label{Lips:final1}
\begin{split}
\int_{\mathbb{R}^{3}}\text{d}p& \,\big(\mathcal{Q}_{q}[f](p) - \mathcal{Q}_{q}[g](p)\big)\big( 1 + |p|^{2} \big)\text{sign}(f-g)\\
&\qquad\leq C\max\big\{m_{2}\big\langle f+g \big\rangle,m_{4}\big\langle f+g \big\rangle,m_{6}\big\langle f + g \big\rangle \big\}\\
&\times \Big(m_{2}\big\langle |f-g| \big\rangle + m_{4}\big\langle |f-g| \big\rangle + m_{6}\big\langle |f-g| \big\rangle\Big)\,.
\end{split}
\end{align}
For the linear part it follows, after explicit calculation of the weak formulation for test function $\varphi(p)=\big(1+ |p|^{2}\big)\text{sign}(f-g)(p)$, that
\begin{align*}
\begin{split}
\int_{\mathbb{R}^3}&\text{d}p\,\big( \mathcal{L}[f](p) - \mathcal{L}[g](p) \big)\varphi(p)\\
&\qquad= \int_{\mathbb{R}_+}\int_{\mathbb{R}_+}\text{d}|p_1|\,\text{d}|p_2|\,\mathcal{K}_0\big(|p_{1}| + |p_{2}|,|p_{1}|,|p_{2}|\big)\\
&\times(f-g)(|p_1|+|p_2|)\Big[ \varphi(|p_1|) + \varphi(|p_2|) - \varphi(|p_1|+|p_2|)\Big]\\
&\qquad\leq\int_{\mathbb{R}_+}\int_{\mathbb{R}_+}d|p_1|\,d|p_2|\,\mathcal{K}_0\big(|p_{1}| + |p_{2}|,|p_{1}|,|p_{2}|\big)\\
&\times \Big|(f-g)(|p_1|+|p_2|)\Big|\,\Big[ |p_1|^{2} + |p_2|^{2} - \big( |p_1| + |p_2| \big)^{2} +1\Big]
\end{split}
\end{align*}
Therefore,
\begin{align}\label{Lips:final2}
\int_{\mathbb{R}^3}\text{d}p\,\big(\mathcal{L}[f](p) - \mathcal{L}[g](p)\big)\varphi(p)\leq c_0\,m_{7}\big\langle |f-g| \big\rangle - c_{2}\,m_{9}\big\langle |f-g| \big\rangle\,.
\end{align}
As a consequence, using estimates \eqref{Lips:final1} and \eqref{Lips:final2}, it follows that
\begin{align}\label{Lips:finalq}
\begin{split}
\int_{\mathbb{R}^{3}}\text{d}p & \,\big(\mathcal{Q}[f](p) - \mathcal{Q}[g](p)\big)\big( 1 + |p|^{2} \big)\text{sign}(f-g)\\
&\qquad\leq C(\mathcal{S})\Big(m_{2}\big\langle |f-g| \big\rangle + m_{7}\big\langle |f-g| \big\rangle\Big) - c_{2}\,m_{9}\,.
\end{split}
\end{align}
Now, writing
\begin{align*}
\tfrac{1}{n_{c}[f](t)}&\mathcal{Q}[f(t)] - \tfrac{1}{n_{c}[g](t)}\mathcal{Q}[g(t)]\\
&= \Big(\tfrac{1}{n_{c}[f](t)} - \tfrac{1}{n_{c}[g](t)} \Big)\mathcal{Q}[f(t)] + \tfrac{1}{n_{c}[g](t)}\Big( \mathcal{Q}[f(t)] - \mathcal{Q}[g(t)] \Big)\,,
\end{align*}
and using that
\begin{equation*}
\Big|\tfrac{1}{n_{c}[f](t)} - \tfrac{1}{n_{c}[g](t)} \Big| \leq \frac{C(\mathcal{S})}{\delta^{3}}\int^{t}_{0}\text{d}s\,m_{2}\big\langle |f(s)-g(s)| \big\rangle + m_{7}\big\langle |f(s)-g(s)| \rangle\,,
\end{equation*}  
together with \eqref{Lips:finalq}, we can derive the estimate
\begin{align}\label{Lips:final3}
\begin{split}
\int_{\mathbb{R}^{3}}&\text{d}p\,\Big(\tfrac{\kappa_0}{n_{c}[f](t)}\mathcal{Q}[f(t)] - \tfrac{\kappa_0}{n_{c}[g](t)}\mathcal{Q}[g(t)]\Big)\varphi (p)\\
&\qquad\leq C_{\delta}(\mathcal{S})\bigg(m_{2}\big\langle |f(t)-g(t)| \big\rangle + m_{7}\big\langle |f(t)-g(t)| \big\rangle\\
&\int^{t}_{0}\text{d}s\,m_{2}\big\langle |f(s)-g(s)| \big\rangle + \int^{t}_{0}\text{d}s\,m_{7}\big\langle |f(s)-g(s)| \big\rangle\bigg)\\
&\qquad\qquad - c(\mathcal{S},n_0)\, m_{9}\big\langle |f(t)-g(t)| \big\rangle\,.
\end{split}
\end{align}
After integrating estimate \eqref{Lips:final3} from $[0,t]$, it follows that 
\begin{align}\label{Lips:final4}
\begin{split}
&\int^{t}\text{d}s\int_{\mathbb{R}^{3}}\text{d}p\,\Big(\tfrac{\kappa_0}{n_{c}[f](s)}\mathcal{Q}[f(s)] - \tfrac{\kappa_0}{n_{c}[g](s)}\mathcal{Q}[g(s)]\Big)\varphi (p)\\
&\leq \int^{t}_{0}\text{d}s\bigg[C_{\delta}(\mathcal{S})(1+T)\Big(m_{2}\big\langle |f(s)-g(s)| \big\rangle + m_{7}\big\langle |f(s)-g(s)| \big\rangle \Big)\\
&- c(\mathcal{S},n_0)\, m_{9}\big\langle |f(s)-g(s)| \big\rangle\bigg] \leq C_{\delta,T}(\mathcal{S})\int^{t}_{0}\text{d}s\,m_{2}\big\langle |f(s)-g(s)| \big\rangle\,.
\end{split}
\end{align}
For the last inequality we used that $$C_{\delta}(\mathcal{S})(1+T)\big(|p|^{2} +  |p|^{7}\big) - c(\mathcal{S},n_0)|p|^{9}\leq C_{\delta,T}(\mathcal{S})|p|^{2}\,.$$  This completes the proof the the one-side Lipschitz property.
\end{proof}

Let us complete now the proof of Theorem \ref{Theorem:ExistenceUniqueness}.  As an application of Theorem \ref{Theorem_ODE_Eq}, where  the three conditions \eqref{Holder_C}, \eqref{subtangent}, and \eqref{one-sided-lip} have been verified in subsections \ref{Section:Lipschitz}, \ref{Subtangent}, and \ref{OnesidedLipschitz}, respectively, it follows that the system \eqref{nc}-\eqref{EBGP} has a unique solution $f\in\mathcal{C}\big([0,T];\mathcal{S}\big)$ where $T$ is any time such that $n_{c}[f](t)\geq\delta$, $t\in[0,T]$.  Clearly, such solution $(f(t), n_{c}[f](t))$ satisfies total conservation of mass $$m_{2}\langle f(t) \rangle + n_{c}[f](t) = m_{2}\langle f_0 \rangle + n_{0}\,,$$ and all conditions of Theorem \ref{BECstab} are satisfied.  Therefore,
\begin{equation*}
\inf_{t}n_{c}[f](t)\geq\delta>0\,.
\end{equation*}
As a consequence, $T>0$ is arbitrary.  This proves Theorem \ref{Theorem:ExistenceUniqueness}.
\end{proof}
\begin{proposition}[Creation of polynomial moments]\label{coro:creation}
Let the pair\\ $0\leq(f,n_{c})\in\mathcal{C}\big([0,\infty);\mathcal{S}\big)\times\mathcal{C}\big([0,\infty)\big)$ be the solution of the system \eqref{BGP} with initial datum $(f_0,n_0)>0$ satisfying condition \eqref{ncthres} for some $\delta>0$.  Then, there exists a constant ${C}_k>0$ that depends only on $k>3$ such that
\begin{equation*}
m_{k}\langle f \rangle(t)\leq \Big(\tfrac{1}{\delta (k-3)}\Big)^{\frac{k-3}{5}}\frac{m_{3}}{t^{\frac{k-3}{5}}} + C_{k}\,m^{\frac{k+1}{4}}_{3}\,,\qquad t>0\,.
\end{equation*}
\end{proposition}
\begin{proof}
Recall estimate \eqref{e-ineq-moments}
\begin{equation*}
\frac{\text{d}}{\text{d}t}m_{k+2}(t)\leq \tfrac{\kappa_0}{n_{c}[f](t)}m^{-\frac{5}{k-1}}_{3}\Big(\tilde{C}_{k}\,m^{\frac{(k+4)(k+3)}{4(k-1)}}_{3} - \tilde{c}_{k}\,m^{\frac{k+4}{k-1}}_{k+2}(t)\Big)\,,
\end{equation*}
for some constants $\tilde{C}_{k}$ and $\tilde{c}_{k}$ depending only on $k>1$.  Since $n_{c}[f](t)>0$, for $t\in[0,\infty)$, is Lipschitz continuous, we can solve uniquely the nonlinear ode 
\begin{equation*}
\alpha'(t) = \tfrac{1}{n_{c}[f](\alpha(t))}\,,\quad t>0,\quad \alpha(0)=0\,.
\end{equation*}
The solution $\alpha(t)$ is strictly increasing.  Thus, we can rescale estimate \eqref{e-ineq-moments} by defining the function $y(t)=m_{k+2}(\alpha(t))$, so that
\begin{equation*}
\frac{\text{d}y}{\text{d}t} \leq \kappa_{o}\,m^{-\frac{5}{k-1}}_{3}\Big(\tilde{C}_{k}\,m^{\frac{(k+4)(k+3)}{4(k-1)}}_{3} - \tilde{c}_{k}\,y^{\frac{k+4}{k-1}}\Big)\,.
\end{equation*}
It is not difficult to prove that a super solution for previous differential inequality is given by
\begin{align*}
Y(t)&=\frac{m_{k+2}(0)}{\Big(1+\frac{k-1}{5}\Big(\frac{m_{k+2}(0)}{m_{3}}\Big)^{\frac{5}{k-1}}t\Big)^{\frac{k-1}{5}}} + C_{k}\,m^{\frac{k+3}{4}}_{3}\\
&\qquad\leq \Big(\tfrac{5}{k-1}\Big)^{\frac{k-1}{5}}\frac{m_{3}}{t^{\frac{k-1}{5}}} + C_{k}\,m^{\frac{k+3}{4}}_{3}\,.
\end{align*}
Hence $y(t)\leq Y(t)$ for all times.  Observe that $\alpha'(t)\leq\frac{1}{\delta}$, this implies that $\delta\,t\leq\alpha^{-1}(t)$.  As a consequence,
\begin{align*}
m_{k+2}(t)\leq Y(\alpha^{-1}(t)) &\leq \Big(\tfrac{5}{k-1}\Big)^{\frac{k-1}{5}}\frac{m_{3}}{\big(\alpha^{-1}(t)\big)^{\frac{k-1}{5}}} + C_{k}\,m^{\frac{k+3}{4}}_{3}\\
&\leq \Big(\tfrac{5}{\delta(k-1)}\Big)^{\frac{k-1}{5}}\frac{m_{3}}{t^{\frac{k-1}{5}}} + C_{k}\,m^{\frac{k+3}{4}}_{3}\,.
\end{align*}
\end{proof}

\section{Mittag-Leffler moments}
\subsection{Propagation of Mittag-Leffler tails}
In this section we are interested in studying the propagation and creation of Mittag-Leffler moments of order $a\in[1,\infty)$ and rate $\alpha>0$ for radially symmetric solutions built in section 5.  This concept of Mittag-Leffler tails was introduced recently in \cite{AlonsoGamba:2015:OML} and it is a generalization of the classical exponential tails for hard potentials in Boltzmann equations.  The creation of exponential tail in the solutions formalize, at least qualitatively, the notion of low temperature regime which is key in the derivation of the model.  We perform the analysis using standard moments $\mathcal{M}_{k}$ stressing that same estimates are valid for line moments since $\mathcal{M}_{k} = |\mathbb{S}^{2}|\,m_{k+2}$ in the context of radially symmetric solutions.  In terms of infinite sums, see \cite{AlonsoGamba:2015:OML}, this is equivalent to control the integral
\begin{equation}\label{MittagLefflerMomentSum}
\int_{\mathbb{R}^3}\text{d}p\,f(t,p)\mathcal{E}_a(\alpha^{a}|p|)=\sum_{k=1}^\infty\frac{\mathcal{M}_{k}(t)\alpha^{ak}}{\Gamma(ak+1)}\,,
\end{equation}
where
\begin{equation}\label{ExtendedMittagLeffler}
\mathcal{E}_a(x):=\sum_{k=1}^\infty\frac{x^k}{\Gamma(ak+1)} \approx e^{x^{1/a}}-1\,,\qquad x\gg1\,.
\end{equation}
For convenience define for any $\alpha>0$ and $a\in[1,\infty)$ the partial sums
\begin{equation*}
\mathcal{E}_a^n(\alpha,t):=\sum_{k=1}^n\frac{\mathcal{M}_{k}(t)\alpha^{ak}}{\Gamma(ak+1)}\quad \text{and} \quad
\mathcal{I}^n_{a,\rho}(\alpha,t):=\sum_{k=1}^n\frac{\mathcal{M}_{k+\rho}(t)\alpha^{ak}}{\Gamma(ak+1)}\,,\quad \rho>0\,.
\end{equation*}
This notation will be of good use throughout this section.
\begin{theorem}[Propagation of Mittag-Leffler tails]\label{Theorem:PropaExpo}
Consider the pair $0\leq(f,n_{c})\in\mathcal{C}\big([0,\infty);\mathcal{S}\big)\times\mathcal{C}\big([0,\infty)\big)$ to be the solution of \eqref{BGP} associated to the initial condition $(f_0,n_0)>0$ satisfying condition \eqref{ncthres} for some $\delta>0$.  Take $a\in[1,\infty)$ and suppose that there exists positive $\alpha_0$ such that
$$\int_{\mathbb{R}^3}dp\,f_0(p)\,\mathcal{E}_a(\alpha_0^a|p|)\leq 1\,.$$
Then, there exists positive constant $\alpha:=\alpha(\mathcal{M}_{1}(0),\alpha_0,a)$ such that
\begin{equation}\label{Theorem:PropaExpo:Eq0}
\sup_{t\geq0}\int_{\mathbb{R}^3}dp\,f(t,p)\,\mathcal{E}_a(\alpha^a|p|)\leq 2\,.
\end{equation}
\end{theorem}
\begin{lemma}[From Ref. \cite{AlonsoGamba:2015:OML}]\label{Lemma:BetaEstimate}
Let $k\geq 3$, then for any $a\in[1,\infty)$, we have 
$$\sum_{i=1}^{\left[\frac{k+1}{2}\right]}{{k}\choose{i}}B\big(ai+1,a(k-i)+1\big)\leq \frac{C_a}{(ak)^{1+a}},$$
where $B(\cdot,\cdot)$ is the beta function.  The constant $C_a>0$ depends only on $a$.
\end{lemma}
\begin{lemma}\label{Lemma:PropaExpo}
Let $\alpha>0$, $a\in[1,\infty)$.  Then, the following estimate holds
\begin{equation}\label{Lemma:PropaExpo:Eq0}
\mathcal{J}:=\sum_{k=k_0}^n\sum_{i=1}^{\left[\frac{k+1}{2}\right]}{{k}\choose{i}}\frac{\mathcal{M}_{i+2}\mathcal{M}_{k-i}\,\alpha^{ak}}{\Gamma(ak+1)} \leq \frac{C_{a}}{(ak_0)^{a}} \,\mathcal{E}_a^n\,\mathcal{I}^n_{a,2}\,,\quad n\geq k_0\geq1\,,
\end{equation}
with universal constant $C_{a}$ depending only on $a$.
\end{lemma}
\begin{proof} 
Using the following identities for the Beta and Gamma functions
\begin{align*}
B(&ai+1,\,a(k-i)+1)\\
&=\frac{\Gamma(ai+1)\,\Gamma(a(k-i)+1)}{\Gamma(ai+1+a(k-i)+1)}=\frac{\Gamma(ai+1)\,\Gamma(a(k-i)+1)}{\Gamma(ak+2)}\,,
\end{align*}
and the identity $\alpha^{ak}=\alpha^{\alpha i}\alpha^{a(k-i)}$, we deduce that
\begin{align}\label{Lemma:PropaExpo:Eq0cc}
\begin{split}
\mathcal{J} = \sum_{k=k_0}^n(ak+1)\sum_{i=1}^{\left[\frac{k+1}{2}\right]}{{k}\choose{i}}&\frac{\mathcal{M}_{i+2}\alpha^{ai}}{\Gamma(ai+1)}\frac{\mathcal{M}_{k-i}\alpha^{a(k-i)}}{\Gamma(a(k-i)+1)}\\
&\qquad\times B(ai+1,a(k-i)+1)\,,
\end{split}
\end{align}
where we used that $\Gamma(ak+2)=(ak+1)\Gamma(ak+1)$.  In addition, each component in the inner sum on the right side of \eqref{Lemma:PropaExpo:Eq0cc} can be bounded as
\begin{align*}
\sum_{i=1}^{\left[\frac{k+1}{2}\right]}&{{k}\choose{i}}\frac{\mathcal{M}_{i+2}\alpha^{ai}}{\Gamma(ai+1)}\frac{\mathcal{M}_{k-i}\alpha^{a(k-i)}}{\Gamma(a(k-i)+1)}B(ai+1,a(k-i)+1)\\
&\leq\sum_{i=1}^{\left[\frac{k+1}{2}\right]}\frac{\mathcal{M}_{i+2}\alpha^{ai}}{\Gamma(ai+1)}\frac{\mathcal{M}_{k-i}\alpha^{a(k-i)}}{\Gamma(a(k-i)+1)}\sum_{j=1}^{\left[\frac{k+1}{2}\right]}{{k}\choose{j}}B(aj+1,a(k-j)+1)\,,
\end{align*}
which implies, by Lemma \ref{Lemma:BetaEstimate}, that
\begin{align}\label{Lemma:PropaExpo:Eq0d}
\begin{split}
\sum_{i=1}^{\left[\frac{k+1}{2}\right]}&{{k}\choose{i}}\frac{\mathcal{M}_{i+2}\alpha^{ai}}{\Gamma(ai+1)}\frac{\mathcal{M}_{k-i}\alpha^{a(k-i)}}{\Gamma(a(k-i)+1)}B(ai+1,a(k-i)+1)\\
&\leq\frac{C_a}{(ak)^{1+a}}\sum_{i=1}^{\left[\frac{k+1}{2}\right]}\frac{\mathcal{M}_{i+2}\alpha^{ai}}{\Gamma(ai+1)}\frac{\mathcal{M}_{k-i}\alpha^{a(k-i)}}{\Gamma(a(k-i)+1)}\,.
\end{split}
\end{align}
Combining \eqref{Lemma:PropaExpo:Eq0cc} and \eqref{Lemma:PropaExpo:Eq0d} yields the estimate on $\mathcal{J}$
\begin{equation}\label{Lemma:PropaExpo:Eq0e}
\mathcal{J}\leq C_a\sum_{k=k_0}^n\frac{ak+1}{(ak)^{1+a}}\sum_{i=1}^{\left[\frac{k+1}{2}\right]}\frac{\mathcal{M}_{i+2}\alpha^{ai}}{\Gamma(ai+1)}\frac{\mathcal{M}_{k-i}\alpha^{a(k-i)}}{\Gamma(a(k-i)+1)}\,.
\end{equation}
Noticing that $\frac{ak+1}{(ak)^{1+a}}\leq\frac{1+a}{a}\frac{1}{(ak_0)^{a}}$ for $k\geq k_{0}$, one concludes from \eqref{Lemma:PropaExpo:Eq0e} that
\begin{align}\label{Lemma:PropaExpo:Eq0f}
\begin{split}
\mathcal{J} &\leq \frac{C'_{a}}{(ak_0)^{a}}\sum_{k=k_0}^n\sum_{i=1}^{\left[\frac{k+1}{2}\right]}\frac{\mathcal{M}_{i+2}\alpha^{ai}}{\Gamma(ai+1)}\frac{\mathcal{M}_{k-i}\alpha^{a(k-i)}}{\Gamma(a(k-i)+1)}\\
&\leq \frac{C'_{a}}{(ak_0)^{a}}\sum_{i=1}^n\frac{\mathcal{M}_{i+2}\alpha^{ai}}{\Gamma(ai+1)}\sum_{i=1}^n\frac{\mathcal{M}_{i}\alpha^{ai}}{\Gamma(ai+1)} \leq \frac{C'_{a}}{(ak_0)^{a}} \mathcal{E}_a^n\,\mathcal{I}^n_{a,2}.
\end{split}
\end{align}
\end{proof}
\begin{lemma}\label{Lemma:EIRelation}
The following control is valid for any $\alpha>0$ and $a\in[1,\infty)$ 
\begin{equation}\label{Lemma:EIRelation:E1}
\mathcal{I}_{a,5}^n(\alpha,t) \geq \frac{1}{\alpha^{5/2}}\mathcal{E}_a^n(\alpha,t)-\frac{1}{\alpha^{2}}\mathcal{M}_1\mathcal{E}_a(\alpha^{a-1/2})\,.
\end{equation}
\end{lemma}
\begin{proof}
Observe that 
\begin{equation*}
\mathcal{I}_{a,5}^n(\alpha,t)=\sum_{k=1}^n\frac{\mathcal{M}_{k+5}(t)\alpha^{ak}}{\Gamma(ak+1)}\geq\sum_{k=1}^n\int_{\{|p|
\geq \frac{1}{\sqrt\alpha}\}}\text{d}p\,\frac{|p|^{k+5}\alpha^{ak}}{\Gamma(ak+1)}f(t,p)\,.
\end{equation*}
Note that in the set $\{|p|\geq\frac{1}{\sqrt\alpha}\}$ one has $|p|^{k+5}\geq\frac{|p|^k}{\alpha^{5/2}}$, therefore
\begin{align*}
&\mathcal{I}_{a,6}^n(\alpha,t) \geq \frac{1}{\alpha^{5/2}}\sum_{k=1}^n\int_{\{|p|
\geq\frac{1}{\sqrt\alpha}\}}\text{d}p\,\frac{|p|^{k}\alpha^{ak}}{\Gamma(ak+1)}f(t,p)\\
&=\frac{1}{\alpha^{5/2}}\bigg(\sum_{k=1}^n\int_{\mathbb{R}^3}\text{d}p\,\frac{|p|^{k}\alpha^{ak}}{\Gamma(ak+1)}f(t,p) - \sum_{k=1}^n\int_{\{|p|<\frac{1}{\sqrt\alpha}\}}\text{d}p\,\frac{|p|^{k}\alpha^{ak}}{\Gamma(ak+1)}f(t,p)\bigg)\,.
\end{align*}
In the set $\{|p|<\frac{1}{\sqrt\alpha}\}$ one has $|p|^k<|p|\alpha^{-(k-1)/2}$, consequently
\begin{align*}\nonumber
\mathcal{I}_{a,5}^n&(\alpha,t)\geq\frac{1}{\alpha^{5/2}}\bigg(\mathcal{E}_a^n(t)-\sum_{k=1}^n\int_{\mathbb{R}^3}\text{d}p\,\frac{\alpha^{-(k-1)/2}\alpha^{ak}}{\Gamma(ak+1)}f(t,p)|p|\bigg)\\
& = \frac{1}{\alpha^{5/2}}\mathcal{E}_a^n(t)- \frac{\mathcal{M}_{1}}{\alpha^{2}}\sum_{k=1}^n\frac{\alpha^{(a-1/2)k}}{\Gamma(ak+1)}\geq \frac{1}{\alpha^{5/2}}\mathcal{E}_a^n(t)- \frac{\mathcal{M}_{1}}{\alpha^{2}}\mathcal{E}_a(\alpha^{a-1/2})\,.
\end{align*}
\end{proof}
\begin{proof}\textbf{(of Theorem \ref{Theorem:PropaExpo})}  The proof consists in showing that for any $a\in [1,\infty)$, there exists positive constant $\alpha$ such that
\begin{equation}\label{Theorem:PropaExpo:0}
\mathcal{E}_a^n(\alpha,t) \leq 2, \quad \forall\, t\geq0,\quad \forall\, n\in \mathbb{N}\backslash\{0\}.
\end{equation}
For this purpose we define for sufficiently small $\alpha>0$, chosen in the sequel, the sequence of times
$$T_n:=\sup\big\{t\geq 0\, \big|\, \mathcal{E}_a^n(\alpha,\tau) \leq 2, \forall\,\tau\in[0,t]\big\}$$
and prove that $T_n=+\infty$.  This sequence of times is well-defined and positive.  Indeed, for any $\alpha\leq\alpha_0$
\begin{align*}
\mathcal{E}^n_a(\alpha,0)&=\sum_{k=1}^n\frac{\mathcal{M}_k(0)\alpha^{ak}}{\Gamma(ak+1)}\leq\sum_{k=1}^n\frac{\mathcal{M}_k(0)\alpha_0^{ak}}{\Gamma(ak+1)}=\int_{\mathbb{R}^{3}}\text{d}p\,f_0(p)\mathcal{E}_a(\alpha_0^a|p|) \leq 1\,.
\end{align*}
Since each term $\mathcal{M}_k(t)$ is continuous in $t$, the partial sum $\mathcal{E}^n_a(\alpha,t)$ is also continuous in $t$. Therefore, $\mathcal{E}^n_a(\alpha,t) \leq 2$ in some nonempty interval $(0,t_{n})$ and, thus, $T_n$ is well-defined and positive for every $n\in\mathbb{N}$.

Now, let us establish a differential inequality for the partial sums that implies $T_n=+\infty$.  Note that 
\begin{equation*}
\tfrac{n_{c}}{\kappa_0}\frac{\text{d}}{\text{d}t} \mathcal{M}_{k} \leq 2\sum_{i=1}^{\left[\frac{k+1}{2}\right]}{{k}\choose{i}}\mathcal{M}_{i+2}\mathcal{M}_{k-i} - c_{k}\,\mathcal{M}_{k+5}\,.
\end{equation*}
Here $c_{k}>0$ was defined in Lemma \ref{control1}.  Multiplying the above inequality by $\frac{\alpha^k}{\Gamma(ak+1)}$ and summing with respect to $k$ in the interval $k_0\leq k \leq n$, with $k_0\geq1$ to be chosen later on sufficiently large,
\begin{align}\label{Theorem:PropaExpo:1}
\begin{split}
\tfrac{n_{c}}{\kappa_0}\frac{\text{d}}{\text{d}t}&\sum_{k=k_0}^n \frac{\mathcal{M}_{k}\,\alpha^k}{\Gamma(ak+1)}\\
&\leq 2\sum_{k=k_0}^n\sum_{i=1}^{\left[\frac{k+1}{2}\right]}{{k}\choose{i}}\frac{\mathcal{M}_{i+2}\mathcal{M}_{k-i}\,\alpha^k}{\Gamma(ak+1)} - c_{k_0}\sum_{k=k_0}^n\frac{\mathcal{M}_{k+5}\,\alpha^k}{\Gamma(ak+1)}\,.
\end{split}
\end{align}
Here we used the fact that $c_{k}$ increases in $k$.  We observe that the sum on the left side of \eqref{Theorem:PropaExpo:1} will become $\tfrac{n_{c}}{\kappa_0}\frac{\text{d}}{\text{d}t}\mathcal{E}^n_a(\alpha,t)$ after adding
\begin{equation}\label{Theorem:PropaExpo:1aa}
\tfrac{n_{c}}{\kappa_0}\frac{\text{d}}{\text{d}t} \sum_{k=1}^{k_0-1}\frac{\mathcal{M}_{k}\,\alpha^k}{\Gamma(ak+1)} \leq C(k_0,\alpha_0,a) < \infty
\end{equation}
to this expression.  The latter inequality holds due to the choice $\alpha\le\alpha_0$ and the control of moments Theorem \ref{Theorem:PropaPolynomial}.  Therefore, from \eqref{Theorem:PropaExpo:1} and \eqref{Theorem:PropaExpo:1aa}, we obtain the differential inequality
\begin{align}\label{Theorem:PropaExpo:1a}
\begin{split}
\tfrac{n_{c}}{\kappa_0}\frac{\text{d}}{\text{d}t}\mathcal{E}^n_a(\alpha,t) \leq 2\sum_{k=k_0}^n&\sum_{i=1}^{\left[\frac{k+1}{2}\right]}{{k}\choose{i}}\frac{\mathcal{M}_{i+2}\mathcal{M}_{k-i}\,\alpha^k}{\Gamma(ak+1)}\\
& - c_{k_{0}}\sum_{k=k_0}^n\frac{\mathcal{M}_{k+5}\,\alpha^k}{\Gamma(ak+1)} + C(k_0,\alpha_0,a).
\end{split}
\end{align}
Let us now estimate the sum on the right side of \eqref{Theorem:PropaExpo:1a}.  Again, we deduce from propagation of moments Theorem \ref{Theorem:PropaPolynomial} that
$$\sum_{k=1}^{k_0}\frac{\mathcal{M}_{k+5}\,\alpha^k}{\Gamma(ak+1)}\leq \sum_{k=1}^{k_0}\frac{\mathcal{M}_{k+5}\,\alpha_0^k}{\Gamma(ak+1)}\leq C(k_0,\alpha_0,a)\,,$$
which leads to the following estimate for \eqref{Theorem:PropaExpo:1a} 
\begin{align}\label{Theorem:PropaExpo:2}
\begin{split}
\tfrac{n_{c}}{\kappa_0}\frac{\text{d}}{\text{d}t}\mathcal{E}^n_a(\alpha,t) \leq 2\sum_{k=k_0}^n&\sum_{i=1}^{\left[\frac{k+1}{2}\right]}{{k}\choose{i}}\frac{\mathcal{M}_{i+2}\mathcal{M}_{k-i}\,\alpha^k}{\Gamma(ak+1)}\\
& - c_{k_{0}}\sum_{k=1}^n\frac{M_{k+5}\,\alpha^k}{\Gamma(ak+1)}+C(k_0,\alpha_0,a)\,.
\end{split}
\end{align}
Therefore, as a consequence of the definition of $\mathcal{I}_{a,5}^n$ and Lemma \ref{Lemma:PropaExpo}
\begin{align}\label{Theorem:PropaExpo:2}
\begin{split}
\tfrac{n_{c}}{\kappa_0}\frac{\text{d}}{\text{d}t} \mathcal{E}^n_a(\alpha,t)\leq 2\sum_{k=k_0}^n&\sum_{i=1}^{\left[\frac{k+1}{2}\right]}{{k}\choose{i}}\frac{\mathcal{M}_{i+2}\mathcal{M}_{k-i}\,\alpha^k}{\Gamma(ak+1)} - c_{k_0}\,\mathcal{I}^n_{a,5} + C(k_0,\alpha_0,a)\\
&\leq \tfrac{2C_{a}}{(ak_0)^{a}}\mathcal{E}^n_{a}\,\mathcal{I}^n_{a,2} - c_{k_{0}}\,\mathcal{I}^n_{a,5}+C(k_0,\alpha_0,a)\,.
\end{split}
\end{align}
We now estimate the right hand side of \eqref{Theorem:PropaExpo:2} starting with the term $\mathcal{I}^n_{a,2}$.  Using Cauchy inequality $|p|^2 \leq \frac{3}{5} +\frac{2}{5}|p|^5\,$, then 
$$\mathcal{M}_{k+2}\leq \tfrac{3}{5}\mathcal{M}_k + \tfrac{2}{5}\mathcal{M}_{k+5}\,,\qquad k\geq0\,.$$
Multiplying this inequality with $\frac{\alpha^{ak}}{\Gamma(ak+1)}$ and summing with respect to $k$ in the interval $0\leq k \leq n$ yields
$$\mathcal{I}_{a,2}^n\leq \tfrac{3}{5}\mathcal{E}_a^n + \tfrac{2}{5}\mathcal{I}_{a,5}^n\leq \tfrac{6}{5} + \tfrac{2}{5}\mathcal{I}_{a,5}^n\,,$$
where the last inequality follows since we are considering $t\in[0,T_n]$ so that $\mathcal{E}_a^n\leq 2$.  Therefore,
\begin{equation}\label{Theorem:PropaExpo:4}
\tfrac{n_{c}}{\kappa_0}\frac{\text{d}}{\text{d}t} \mathcal{E}^n_{a}
\leq \tfrac{5\,C_{a}}{(ak_0)^{a}}\big( 1 + \tfrac{1}{3}\mathcal{I}^n_{a,5}\big) - c_{k_{0}}\mathcal{I}^n_{a,5} + C(k_0,\alpha_0,a)\,.
\end{equation}
Choosing $k_0:=k_{0}(a)$ sufficiently large, the term $\frac{5C_{a}}{3(ak_0)^{a}}\mathcal{I}^n_{a,5}$ is absorbed by $\frac{ c_{k_{0}} }{2}\mathcal{I}_{a,5}^n$.  Thus,
\begin{equation}\label{Theorem:PropaExpo:4}
\tfrac{n_{c}}{\kappa_0}\frac{\text{d}}{\text{d}t} \mathcal{E}^n_{a}
\leq -\tfrac{c_{k_{0}}}{2}\mathcal{I}^n_{a,5}+C(\mathcal{M}_{1},\alpha_0,a)\,.
\end{equation}
Estimating the right  side of \eqref{Theorem:PropaExpo:4} in terms of $\mathcal{E}^n_{a}$ using Lemma \ref{Lemma:EIRelation},  it is concluded that
\begin{equation*}
\tfrac{n_{c}}{\kappa_0}\frac{\text{d}}{\text{d}t} \mathcal{E}^n_{a}
\leq - \tfrac{c_{k_{0}}}{2\alpha^{5/2}}\mathcal{E}^n_{a}+\tfrac{c_{k_{0}}}{2\alpha^{2}}\,\mathcal{M}_1\,\mathcal{E}_a(\alpha^{a-1/2}) + C(\mathcal{M}_{1},\alpha_0,a)\,.
\end{equation*}
Therefore, one has that for $t\in[0,T_{n}]$
\begin{equation}\label{Theorem:PropaExpo:5}
\mathcal{E}_a^n \leq \max\Big\{1,\tfrac{2\alpha^{5/2}}{c_{k_{0}}}\Big(\tfrac{c_{k_{0}}}{2\alpha^{2}}\,\mathcal{M}_1\,\mathcal{E}_a(\alpha^{a-1/2})+C(\mathcal{M}_{1},\alpha_0,a)\Big)\Big\} < 2\,,
\end{equation}
provided that $\alpha:=\alpha(\mathcal{M}_{1},\alpha_0,a)>0$ is sufficiently small, for instance such that
\begin{equation*}
\tfrac{2\alpha^{5/2}}{c_{k_{0}}}\left(\tfrac{c_{k_{0}}}{2\alpha^{2}}\mathcal{M}_1\mathcal{E}_a(\alpha^{a-1/2}) + C(\mathcal{M}_{1},\alpha_0,a)\right)<2\,.
\end{equation*}
Given the continuity of $\mathcal{E}_a^n(\alpha,t)$ with respect to $t$, estimate \eqref{Theorem:PropaExpo:5} readily implies that $T_n = +\infty$.  Therefore, $\mathcal{E}^n_a(\alpha,t)\leq 2$ for $t\geq0$ and $n\in\mathbb{N}\backslash\{0\}$. Now taking the limit as $n\to\infty$ and using the definition 
of  Mittag-Leffler moments of order $a\in[1,\infty)$ and rate $\alpha>0$, as defined in \eqref{MittagLefflerMomentSum},  yields
$$\int_{\mathbb{R}^3}\text{d}p\,f(t,p)\,\mathcal{E}_a(\alpha^a|p|)=\lim_{n\to\infty}\mathcal{E}_a^n(\alpha,t)\leq 2\,.$$ 
This concludes the argument. 
\end{proof}
\subsection{Creation of exponential tails}
\begin{theorem}\label{Theorem:CreationExpo}
Let the pair $0\leq(f,n_{c})\in\mathcal{C}\big([0,\infty);\mathcal{S}\big)\times\mathcal{C}\big([0,\infty)\big)$ be the solution of \eqref{BGP}.  Assume that $(f_0,n_0)>0$ is such that condition \eqref{ncthres} is satisfied for some $\delta>0$.  Then, there exists a constant $\alpha>0$ depending on $m_{2}(0)$, $m_{3}$, $n_0$, and $\delta>0$, such that
\begin{equation}\label{Theorem:PropaExpo:Eq0}
\int_{\mathbb{R}^3}dp\,f(t,p)|p|e^{\alpha\,\min\{1,t^{\frac{1}{5}}\}|p|} \leq \frac{1}{2\alpha},\quad\forall\, t>0.
\end{equation}
\end{theorem}
\begin{proof}
Thanks to Corollary \ref{coro:creation}, the moments of $f(t)$ enjoy the estimate
\begin{equation*}
m_k(t)\leq C_k(\delta,m_3)\Big(t^{-\frac{k-3}{5}} + 1 \Big), \ \ \forall \, k>3\,.
\end{equation*}
This implies that for any $0\leq t\leq 1$
\begin{equation}\label{CreationEq1}
\mathcal{E}^{n}_{1}(t^{\frac{1}{5}}\alpha,t)=\int_{\mathbb{R}^{3}}\text{d}p\,f(t,p)\mathcal{E}^{n}_{1}\big(t^{\frac{1}{5}}\alpha|p|\big)\leq C_{n}(\alpha)\,t^{\frac{1}{5}}\,, \quad \alpha>0\,.
\end{equation}
Fix parameters $\alpha, \vartheta\in(0,1]$ and define
\begin{equation*}
 T_{n}:=\sup\Big\{t\in(0,1] \big| \mathcal{E}^{n}_{1}(t^{\frac{1}{5}}\alpha,t) \leq t^{\frac{1-\vartheta}{5}}\Big\}\,.
\end{equation*}
We proof that for sufficiently small $\alpha>0$ depending only on the initial data (through $m_{2}(0)$, $m_{3}$, and $n_0$), it holds that $T_{n}=1$ for all $n\in\mathbb{N}$ and $\vartheta\in(0,1]$.  One notices first that $T_{n}>0$ for each $n$ thanks to \eqref{CreationEq1}.  Also, for $n\geq k_{0}\geq 1$ we have that
\begin{align}\label{CreationEq2}
\begin{split}
\frac{\text{d}}{\text{d}t}&\sum^{n}_{k=k_0} \mathcal{M}_{k}(t)\frac{(t^{\frac{1}{5}}\alpha)^{k}}{k!}\\
&= \sum^{n}_{k=k_0} \mathcal{M}'_{k}(t)\frac{(t^{\frac{1}{5}}\alpha)^{k}}{k!} + \frac{\alpha}{5t^{\frac{4}{5}}}\sum^{n}_{k=k_0} \mathcal{M}_{k}(t)\frac{(t^{\frac{1}{5}}\alpha)^{k-1}}{(k-1)!}\,. 
\end{split}
\end{align}
Observe that for the last term in the right side of \eqref{CreationEq2}
\begin{align*}
\frac{\alpha}{5t^{\frac{4}{5}}}\sum^{n}_{k=k_0} &\mathcal{M}_{k}(t)\frac{(t^{\frac{1}{5}}\alpha)^{k-1}}{(k-1)!} \\
&= \frac{\alpha}{5t^{\frac{4}{5}}}\sum^{n}_{k=k_0+5} \mathcal{M}_{k}(t)\frac{(t^{\frac{1}{5}}\alpha)^{k-1}}{(k-1)!} + \frac{\alpha}{5t^{\frac{4}{5}}}\sum^{k_0+5}_{k=k_0} \mathcal{M}_{k}(t)\frac{(t^{\frac{1}{5}}\alpha)^{k-1}}{(k-1)!}\\
&= \frac{\alpha^{5}}{5}\sum^{n-5}_{k=k_{0}}\mathcal{M}_{k+5}(t)\frac{(t^{\frac{1}{5}}\alpha)^{k}}{(k+4)!} + \frac{\alpha}{5t^{\frac{4}{5}}}\sum^{k_0+5}_{k=k_0} \mathcal{M}_{k}(t)\frac{(t^{\frac{1}{5}}\alpha)^{k-1}}{(k-1)!}\\
&\leq \frac{\alpha^{5}}{5}\sum^{n}_{k=k_{0}}\mathcal{M}_{k+5}(t)\frac{(t^{\frac{1}{5}}\alpha)^{k}}{k!} + \frac{\alpha^{k_0}}{t^{\frac{4}{5}}}C(k_0,m_{3})\,,\quad 0<\alpha\leq1\,.
\end{align*}
Thus, arguing as in \eqref{Theorem:PropaExpo:1}-\eqref{Theorem:PropaExpo:2} we conclude that for the quantities
\begin{equation*}
\mathcal{E}^{n}_{1}:=\mathcal{E}^{n}_{1}(t^{\frac{1}{5}}\alpha,t)\,,\quad \mathcal{I}^{n}_{1,5}:=\mathcal{I}^{n}_{1,5}(t^{\frac{1}{5}}\alpha,t)\,,
\end{equation*}
it follows that
\begin{equation}\label{CreationEq3}
\frac{\text{d}}{\text{d}t}\mathcal{E}^{n}_{1}\leq \tfrac{C\,\kappa_0}{k_{0}\,n_{c}[f](t)}\mathcal{E}^{n}_{1}\,\mathcal{I}^{n}_{1,2} - \big(\tfrac{\kappa_0}{n_{c}[f](t)}c_{k_{0}} - \tfrac{\alpha^{5}}{5} \big)\mathcal{I}^{n}_{1,5} + \frac{\alpha}{t^{\frac{4}{5}}}C(k_0,m_{3})\,.
\end{equation}
Using that $\mathcal{I}^{n}_{1,2}\leq \mathcal{E}^{n}_{1} + \mathcal{I}^{n}_{1,5}\leq 1+ \mathcal{I}^{n}_{1,5}$, recalling the definition of $T_{n}$, it follows from \eqref{CreationEq3}
\begin{align}\label{CreationEq4}
\begin{split}
&\frac{\text{d}}{\text{d}t}\mathcal{E}^{n}_{1} \leq \tfrac{C\,\kappa_0}{k_{0}\,n_{c}[f](t)}\\
& - \Big(\tfrac{\kappa_0}{n_{c}[f](t)}\big(c_{k_{0}} - \tfrac{C}{k_0}\big) - \tfrac{\alpha^{5}}{5} \Big)\mathcal{I}^{n}_{1,5} + \frac{\alpha}{t^{\frac{4}{5}}}C\big(k_0,m_{3}\big)\,,\quad 0 < t \leq T_{n}\,.
\end{split}
\end{align}
Now, fix $k_0\in\mathbb{N}$ sufficiently large and, then, $\alpha\in(0,1]$ sufficiently small such that 
\begin{equation*}
c_{k_{0}} - \tfrac{C}{k_0} \geq \tfrac{c_{k_{0}}}{2}\,,\qquad \tfrac{\alpha^{5}}{5}\leq \frac{\kappa_0\,c_{k_0}}{4(m_{2}(0) + n_0)}\leq \frac{\kappa_0\,c_{k_0}}{4n_{c}[f](t)}\,, 
\end{equation*}
to conclude from \eqref{CreationEq4} that

\begin{equation}\label{CreationEq5}
\frac{\text{d}}{\text{d}t}\mathcal{E}^{n}_{1}\leq \tfrac{C\,\kappa_0}{k_{0}\,\delta} - \frac{\kappa_0\,c_{k_0}}{4(m_{2}(0) + n_0)}\mathcal{I}^{n}_{1,5} + \frac{\alpha}{t^{\frac{4}{5}}}C(k_0,m_{3})\,,\quad 0 < t \leq T_{n}\,.
\end{equation}
Also, observe that
\begin{align*}
\mathcal{I}^{n}_{1,5} &= \sum^{n}_{k=1}\mathcal{M}_{k+5}(t)\frac{(t^{\frac{1}{5}}\alpha)^{k}}{k!} \\
&= \frac{1}{t\alpha^{5}}\sum^{n+5}_{k=6}\mathcal{M}_{k}(t)\frac{(t^{\frac{1}{5}}\alpha)^{k}}{(k-5)!}\geq \frac{1}{t\alpha^{5}}\sum^{n}_{k=6}\mathcal{M}_{k}(t)\frac{(t^{\frac{1}{5}}\alpha)^{k}}{k!}\\
& = \frac{1}{t\alpha^{5}}\mathcal{E}^{n}_{1} - \frac{1}{t\alpha^{5}}\sum^{5}_{k=1}\mathcal{M}_{k}(t)\frac{(t^{\frac{1}{5}}\alpha)^{k}}{k!}\geq \frac{1}{t\alpha^{5}}\mathcal{E}^{n}_{1} - \frac{C(m_{3})}{t^{\frac{4}{5}}\alpha^{4}}\,.
\end{align*}
Together with \eqref{CreationEq5}, this leads finally to
\begin{equation*}
\frac{\text{d}}{\text{d}t}\mathcal{E}^{n}_{1}\leq \frac{C(m_2(0),n_0,m_{3},\delta)}{t^{\frac{4}{5}}\alpha^{4}} - \frac{c(m_2(0),n_0,m_{3})}{t\alpha^{5}}\mathcal{E}^{n}_{1}\,,\quad 0<t\leq T_{n}\,.
\end{equation*}
A simple integration of this differential inequality shows that choosing $\alpha>0$ sufficiently small, say
\begin{equation*}
\frac{\alpha}{c+\alpha^{5}/5}\frac{C}{c}<1, 
\end{equation*}
implies that $\mathcal{E}^{n}_{1}<t^{\frac{1}{5}}$.  That is,
\begin{equation*}
 \int_{\mathbb{R}^{3}}\text{d}p\,f(t,p)\mathcal{E}^{n}_{1}(t^{\frac{1}{5}}\alpha|p|)<t^{\frac{1}{5}}\,,\quad 0 \leq t\leq T_{n}\,.
\end{equation*}
Time continuity of $\mathcal{E}^{n}_{1}$ and the maximality of $T_{n}$ imply that $T_{n}=1$ for all $n\geq 1$ and $\vartheta\in(0,1]$.  In particular, sending $\vartheta\rightarrow0$ and, then, $n\rightarrow\infty$ one arrives to
\begin{equation*}
\int_{\mathbb{R}^{3}}\text{d}p\,f(t,p)\mathcal{E}_{1}(t^{\frac{1}{5}}\alpha|p|)\leq t^{\frac{1}{5}}\,,\qquad 0 \leq t \leq 1\,.
\end{equation*}
The result follows after noticing that
\begin{equation*}
\mathcal{E}_{1}(t^{\frac{1}{5}}\alpha|p|)\geq t^{\frac{1}{5}}\alpha |p| e^{t^{\frac{1}{5}}\frac{\alpha}{2}|p|},\qquad 0\leq t\leq 1,
\end{equation*}
and recalling that, after creation, exponential tails will uniformly propagate thanks to Theorem \ref{Theorem:PropaExpo}.
\end{proof}
\\\\
\textbf{Acknowledgements.} 
This work has been partially supported by NSF grants DMS 143064 and  RNMS (Ki-Net) DMS-1107444.   The authors would like to thank Professor Daniel Heinzen, Professor Linda Reichl, Professor Mark Raizen and Professor Robert Dorfman for fruitful discussions on the topic.  Support from  the Institute of Computational Engineering and Sciences (ICES) at the University of Texas Austin is gratefully acknowledged.  The research was partially carried on while M.-B. Tran and R. Alonso were visiting ICES.  
\section{Appendix: Proof of Theorem~\ref{Theorem_ODE}}\label{Appendix}
{The proof follows the same lines of the argument of  Bressan's proof of Theorem A.1 in \cite{Bressan} with suitable modifications to deal with causal operators. The proof is divided into three steps:\\\\
\noindent
\textbf{Step 1.(Extension)}
Take $u\in\mathcal{C}\big([0,t];\mathcal{S}\big)$, for any fixed $t\in[0,T)$.  Using the fact that $\mathcal{S}$ is bounded, the causality of $Q(\cdot)$, and the uniform H\"{o}lder estimate
\begin{equation*}
\sup_{s\in[0,t]}\big\|Q(u)(s)\big\|\leq C\sup_{s\in[0,t]}\big\|u(s)\big\|^{\beta} \leq C\,C^{\beta}_{\mathcal{S}}\,.
\end{equation*}
Thanks to the uniform sub-tangent condition, for any fixed $\varepsilon\in(0,1)$ there exists $h(u,\varepsilon)>0$ such that 
\begin{equation*}
B\big(u(t)+h\,Q(u)(t),\varepsilon\big)\cap {\mathcal{S}}\backslash\big\{u(t)+hQ(u)(t)\big\}\neq\emptyset\,,\quad\forall\, h\in\big(0, h(u,\varepsilon)\big]\,.
\end{equation*}
We fix $h>0$ as $h:=\min\{1,(\varepsilon/2C)^{\frac{1}{\beta}}(C^{\beta}_{S}+2)^{-1},h(u,\varepsilon)\}$.  As a consequence, there exists $w$ in such set satisfying
\begin{equation*}
\big\|w - u(t) - h\,Q(u)(t)\big\|\leq \tfrac{\varepsilon h}{2}\,.
\end{equation*}
Consider now the linear map
\begin{equation*}
s\mapsto \rho(s)=u(t)+\frac{(s-t)\big(w-u(t)\big)}{h},\quad s\in[t,t+h]\,.
\end{equation*}
Give the fact that the set $\mathcal{S}$ is convex and closed, $\rho(s) \in \mathcal{S}$ for all $s\in[t,t+h]$. Moreover, since the right derivative is $\dot{\rho}(s)=\frac{w-u(t)}{h}$ in $[t,t+h)$, it follows that
\begin{equation*}
\big\|\dot{\rho}(s) - Q(u)(t)\big\|\leq \tfrac{\varepsilon}{2}\,,\quad s\in[t,t+h)\,.
\end{equation*}
Also, we observe that 
\begin{align}\label{Bre2}
\begin{split}
\big\|\rho(s)-u(t)\big\|=\Big\|\tfrac{(s-t)(w-u(t))}{h}\Big\|& \leq \big\|w-u(t)\big\|\\
&\hspace{-1cm}\leq h\big\|Q(u)(t)\big\|+\tfrac{\varepsilon h}{2}\leq h\big(C_{\mathcal{S}}^\beta + 1\big)\,.
\end{split}
\end{align}
Define now the extension $u_{e}\in\mathcal{C}\big([0,t+h];\mathcal{S}\big)$ as
\begin{equation*}
u_{e}(s)=\bigg\{
\begin{array}{cl}
u(s) & \text{for}\;s\in[0,t)\,,\\
\rho(s) & \text{for}\;s\in[t,t+h]\,,
\end{array}
\end{equation*}
Then, for any $0\leq t \leq s \in[t,t+h)\subset[0,T]$, the uniform H\"{o}lder continuity property of $Q$ and estimate \eqref{Bre2} imply that
\begin{align*}
&\big\|Q(u_{e})(s) - Q(u)(t)\big\|\\
&\hspace{-0.5cm}\leq C\Big(\sup_{\sigma\in[0,t]}\big\|u_{e}(\sigma) - u(\sigma)\big\|^{\beta} + \big\|u_{e}(s) - u(t)\big\|^{\beta} + |s-t|^{\beta}\Big)\\
& \le C\Big(\big\|\rho(s) - u(t)\big\|^{\beta} + h^{\beta}\Big) \leq C\big(C^{\beta}_{\mathcal{S}} + 2\big)h^{\beta} \leq \tfrac{\varepsilon}{2}\,.
\end{align*}
Therefore, for the extension $u_{e}$ follows that in the interval $s\in[t,t+h)$
\begin{align}\label{Theorem:ODE:E1}
\begin{split}
\big\|\dot{u}_{e}(s) &- Q(u_{e})(s)\big\| = \big\|\dot{\rho}(s) - Q(u)(t) + Q(u)(t) - Q(u_{e})(s)\big\|\\
&\leq\big\|\dot{\rho}(s) - Q(u)(t)\big\| + \big\|Q(u)(t) - Q(u_{e})(s)\big\| \leq \varepsilon\,.
\end{split}
\end{align}
And, as consequence of this fact
\begin{equation}\label{Theorem:ODE:E2}
\begin{split}
\sup_{s\in[t,t+h]}\big\|\dot{u}_{e}(s)\big\|&\leq 1+\sup_{s\in[0,t+h]}\big\|Q(u_{e})(s)\big\|\\
&\leq 1 + C\sup_{s\in[0,t+h]}\big\|u_{e}(s)\big\|^{\beta}\leq 1+ C\,C^{\beta}_{\mathcal{S}}\,. 
\end{split}
\end{equation}
valid for any $\varepsilon\in(0,1)$.\\

\noindent
\textbf{Step 2.(Piecewise approximations)} Fix $\varepsilon\in(0,1)$.  Starting from $t=0$ we use the extension procedure of Step 1 to construct a piecewise linear function $\rho:=\rho^{\varepsilon}\in\mathcal{C}\big([0,\tau);S\big)$ satisfying the estimates
\begin{equation}\label{step2}
\sup_{s\in[0,\tau)}\big\|\dot{\rho}(s) - Q(\rho)(s)\big\| \leq \varepsilon\,,\qquad \sup_{s\in[0,\tau)}\big\|\dot{\rho}(s)\big\|\leq C\,,
\end{equation}
with initial condition $\rho^{\varepsilon}(0)=u_0$.

\noindent
Suppose that $\rho$ is constructed on a series of intervals $[0,\tau_1]$, $[\tau_1,\tau_2]$, $\cdots$, $[\tau_n,\tau_{n+1}]$, $\cdots$. Moreover, suppose the increasing sequence $\{\tau_n\}$ is bounded and, set
\begin{equation*}
\tau=\lim_{n\to\infty}\tau_n\,.
\end{equation*}
Since $\dot{\rho}$ is uniformly bounded, the sequence $\{\rho(\tau_{n})\}$ has a limit. Therefore, we can define $\rho(\tau)$ as
\begin{equation*}
\rho(\tau)=\lim_{n\to\infty}\rho(\tau_n)\,.
\end{equation*}
This implies that $\rho$ is, in fact, defined on $[0,\tau]$.  It also implies, by the extension procedure of Step 1, that $\tau=T$. 

\noindent
\textbf{Step 3.(Limit)} Let us now consider two sequences of approximate solutions $u^\varepsilon$, $w^\varepsilon$, where $\varepsilon$ tends to $0$.  From Step 1 and Step 2, one can see that the time interval $[0,T]$ can be decomposed into $$\left(\bigcup_{\gamma}I_\gamma\right)\bigcup \mathfrak{N},$$ where $I_\gamma$ are countably many open intervals where $u^\varepsilon$, $w^\varepsilon$ are affine, and $\mathfrak{N}$ is of measure $0$.  Thus, we can take the derivative of the difference $\big\|u^\varepsilon(t)-w^\varepsilon(t)\big\|$ gives
\begin{align*}
\frac{d}{dt}\big\|u^\varepsilon(t)-&w^\varepsilon(t)\big\|=\Big[u^\varepsilon(t)-w^\varepsilon(t),\dot{u}^\varepsilon(t)-\dot{w}^\varepsilon(t)\Big]\\
&\leq \Big[u^\varepsilon-w^\varepsilon,Q(u^\varepsilon)(t)-Q(w^\varepsilon)(t)\Big] + 2\,C\,\varepsilon\\
\end{align*}
In the last inequality we used the first estimate in \eqref{step2}.  Integrating and using the one-sided Lipschitz property
\begin{equation*}
\big\|u^\varepsilon(t)-w^\varepsilon(t)\big\| \leq L\int^{t}_0\text{d}s\big\|u^\varepsilon(s)-w^\varepsilon(s)\big\|  +2\,C\,t\,\varepsilon,
\end{equation*}
which yields, by Gronwall's lemma, that
\begin{equation*}
\big\|u^\varepsilon(t)-w^\varepsilon(t)\big\|\leq \tfrac{2CT}{L}e^{LT}\varepsilon\,.
\end{equation*}
As a consequence, the sequence $\{u^\varepsilon\}$ is Cauchy and converges uniformly to a continuous limit $u\in\mathcal{C}\big([0,T];\mathcal{S}\big)$. Cleary, the function $u$ is \textit{the} solution of our equation.
\bibliographystyle{plain}\bibliography{QuantumBoltzmann}
\end{document}